\documentclass[10pt]{article}

\usepackage{amsfonts} 
\usepackage{amsmath,hhline,latexsym,mathrsfs,amsthm} 
\usepackage{amssymb}
\usepackage{relsize}
\DeclareMathAlphabet\mathbfcal{OMS}{cmsy}{b}{n}
\usepackage{mathtools}
\usepackage{dsfont}
\usepackage{marginnote}
\usepackage{listings}
\usepackage{algorithm2e}
\RestyleAlgo{ruled}

\usepackage{tikz}

\usepackage{hyperref,enumitem}
\usepackage{cleveref}

\usepackage{subcaption}
\usepackage{xcolor}
\usepackage{color}
\usepackage{graphicx}
\usepackage{comment}

\renewcommand{\leq}{\leqslant}

\renewcommand{\geq}{\geqslant}

\usepackage[english]{babel} 
\usepackage{geometry}
\usepackage{xcolor}
\usepackage{color}
\usepackage[colored]{shadethm}
\definecolor{shadethmcolor}{gray}{0.85}
\definecolor{shaderulecolor}{gray}{0.85}

\definecolor{asparagus}{rgb}{0.53, 0.66, 0.42}
\definecolor{bananamania}{rgb}{0.98, 0.91, 0.71}
\definecolor{camel}{rgb}{0.76, 0.6, 0.42}
\definecolor{lavenderblue}{rgb}{0.8, 0.8, 1.0}
\definecolor{deepsaffron}{rgb}{1.0, 0.6, 0.2}

\def\dx{\,\textnormal{d}x}
\def\dt{\textnormal{d}t}

\def\ds{\textnormal{d}s}
\def\d{\,\textnormal{d}}

\title{An observability estimate for the wave equation and applications to the Neumann boundary controllability for semi-linear wave equations}

\author{
 Sue Claret\thanks{Universit\'e Clermont Auvergne, CNRS, LMBP, F-63000 Clermont-Ferrand, France; sue.claret@uca.fr.}
} 
 
\begin{document}
\maketitle 

\newshadetheorem{theoreme}{Theorem}
\newshadetheorem{lemme}{Lemma}
\newshadetheorem{proposition}{Proposition}
\newshadetheorem{remark}{Remark}
\newshadetheorem{corollaire}{Corollary}
\newshadetheorem{definition}{Definition}
\newshadetheorem{def/prop}{Definition/Proposition}
\newshadetheorem{It1}{Theorem 8 - item }
\newshadetheorem{problem}{Problem}

\newshadetheorem{LSalgo}{Least-square algorithm}

\begin{abstract}

We give a boundary observability result for a $1$d wave equation with a potential. We then deduce with a Schauder fixed-point argument the existence of a Neumann boundary control for a semi-linear wave equation $\partial_{tt}y - \partial_{xx}y + f(y) = 0$ under an optimal growth assumption at infinity on $f$ of the type $s\ln^2s$. Moreover, assuming additional assumption on $f'$, we construct a minimizing sequence which converges to a control. Numerical experiments illustrate the results.  This work extends to the Neumann boundary control case the work of Zuazua in $1993$ and the work of Münch and Trélat in $2022$.
\end{abstract}
		


\section{Introduction and main results}

Let $\Omega := (0,1)$, $T>0$ and $Q_T := \Omega\times (0,T)$. We define the Hilbert space $H^1_{(0)}(\Omega) := \left\{ z\in H^1(\Omega); \ z(0) = 0\right\}$ endowed with the norm $\left\|\cdot\right\|_{H^1_{(0)}(\Omega)} := \left\|\partial_x\ \cdot \right\|_{L^2(\Omega)}$. We denote by $H^{-1}_{(0)}(\Omega)$ the dual space of $H^1_{(0)}(\Omega)$ equipped with the dual norm \begin{equation*}
\|w\|_{H^{-1}_{(0)}(\Omega)}:= \sup\limits_{z\in H^1_{(0)}(\Omega)\backslash\{0\}}\frac{\langle w, z\rangle_{-1,1}}{\|z\|_{H^1_{(0)}(\Omega)}},
\end{equation*} where $\langle \cdot, \cdot \rangle_{-1, 1}$ denotes the dual product between $H^{-1}_{(0)}(\Omega)$ and $H^1_{(0)}(\Omega)$. For any $A\in L^\infty(Q_T)$ and $B\in L^2(Q_T)$, we consider the following linear wave equation
\begin{equation}\tag{$\star$}\label{system_adjoint}
\left\lbrace 
    	\begin{aligned}
    	&\partial_{tt} \varphi - \partial_{xx}\varphi + A\varphi = B, &  Q_T, 
    	\\ &\varphi(0,t)=0, \ \partial_x\varphi(1,t)= 0, & (0,T),
    	\\ & \big( \varphi(\cdot,0),\partial_t \varphi(\cdot, 0)\big)  = (\varphi_0, \varphi_1), & \Omega,
    	\end{aligned}
    	\right. 
\end{equation}
where $\varphi = \varphi(x,t)$ is a state and $(\varphi_0, \varphi_1)\in L^2(\Omega)\times H^{-1}_{(0)}(\Omega)$ is a given initial data.  Then, \eqref{system_adjoint} admits a unique solution in the sense of transposition in $\mathcal{C}^0\big([0,T]; L^2(\Omega)\big)\cap\mathcal{C}^1\big([0,T]; H^{-1}_{(0)}(\Omega)\big)$, we refer to Definition \ref{def_sol_transpo} and Theorem \ref{Theorem_reg_donnee_non_reg}.
This paper is devoted to the boundary observability problem corresponding to \eqref{system_adjoint}. Precisely, our main result is as follows

\begin{theoreme}\label{Main_result_Obs_inegality}
Let $T>2$, $A\in L^\infty(Q_T)$ and $B\in L^2(Q_T)$. Then, for any $\varphi$ solution in the sense of transposition of \eqref{system_adjoint}, there exists $C_\text{obs} = C(\Omega, T, A)>0$ such that
\begin{equation}\tag{$$Obs$$}\label{In.obs.g}
\left\|\left(\varphi(\cdot, 0),\partial_t\varphi(\cdot, 0)\right)\right\|^2_{L^2(\Omega)\times H^{-1}_{(0)}(\Omega)}\leq C_{\text{obs}} \left( \left\| \varphi(1,\cdot)\right\|^2_{L^2(0,T)} +\left\| \partial_{tt}\varphi - \partial_{xx}\varphi + A\varphi \right\|^2_{L^2(Q_T)} \right).
\end{equation}
Moreover, there exists $C = C(\Omega, T)>0$ such that
\begin{equation*}
C_{\text{obs}} = Ce^{C\sqrt{\|A\|_{L^\infty(Q_T)}}}.
\end{equation*}
\end{theoreme}
Estimate $\eqref{In.obs.g}$ has been known for many years, see \cite[Theorem $3.1$ $(b)$ p. $271$ or Theorem $3.4$ $(b)$ p. $274$]{lasiecka1989exact} , \cite[Chapter $19$, Lemma $19.A.1$ p. $324$]{leugering2005control} or \cite[Proposition $2.60$, $2.61$ p. $74$]{coron2007control}. In the above results, the constant $C_{\text{obs}}$ depends on the potential $A$. However, the explicit estimate of $C_{\text{obs}}$ as a function of the potential $A$ as in \cite[Theorem $4$ p. $120$]{zuazua1993exact} or \cite[Theorem $3.1$, $3.2$]{doi:10.1137/S0363012999350298} is a part of the problem and it is exactly the main novelty of this paper. The proof is based on the same method as in \cite[Section $3$]{zuazua1993exact}. In particular, the one-dimensional character is used in a fundamental way. 

By duality arguments, boundary observability estimates are equivalent to boundary controllability properties. In particular here, one motivation for Theorem \ref{Main_result_Obs_inegality}, is the exact boundary controllability of the following semi-linear wave equation
\begin{equation}\tag{$\star\star$}\label{system_semilinear}
\left\lbrace 
    	\begin{aligned}
    	&\partial_{tt} y - \partial_{xx}y + f(y) = 0, &  Q_T, 
    	\\ &y(0,\cdot)=0, \ \partial_xy(1,\cdot)= v, & (0,T),
    	\\ & \big( y(\cdot,0),\partial_t y(\cdot, 0)\big)  = (u_0, u_1), & \Omega,
    	\end{aligned}
    	\right. 
\end{equation}
where $(u_0, u_1)\in H^1_{(0)}(\Omega)\times L^2(\Omega)$ is a given initial data, $v$ is a control function and $f\in\mathcal{C}^1(\mathbb{R})$ is a non-linear function. The problem of exact controllability associated with $(\star\star)$ is the following one:

\vspace{0.3cm}

\textit{Given a controllability time $T>0$ large enough, for any initial data $(u_0, u_1)$ and any final data $(z_0, z_1)$ in $H^1_{(0)}(\Omega)\times L^2(\Omega)$, find a control function $v\in L^2(0,T)$ and $y\in \mathcal{C}^0\big([0,T]; H^1_{(0)}(\Omega)\big)\cap \mathcal{C}^1\big([0,T]; L^2(\Omega)\big)$ solution of $(\star\star)$ such that
\begin{equation}\label{Target}
\big( y(\cdot, T), \partial_ty(\cdot, T)\big) = (z_0, z_1) \hspace{1cm} \text{ in } \Omega.
\end{equation}}

\vspace{0.3cm}

The controllability problem of the non-linear wave equation has essentially been studied in the case of distributed control with mainly Dirichlet boundary conditions. The first work on the exact controllability of a finite-dimensional non-linear wave equation is due to Markus in \cite{doi:10.1137/0303008} by using an implicit function theorem. The above method was then applied to obtain local controllability results on non-linear wave equation in \cite{Fattorini1975LocalCO} or \cite{doi:10.1137/0314002}. After that, the global controllability problem for semi-linear wave equation 
is studied under the following growth assumption on the non-linearity :
\begin{equation}\tag{$$H$$}\label{Assumption_on_f_with_p}
\exists \beta>0 \text{ small enough such that } \limsup\limits_{|r|\rightarrow +\infty} \frac{\left|f(r)\right|}{|r|\ln^p|r|} \leq \beta \text{ for some } p\geq0.
\end{equation}
In particular, by a Schauder fixed-point argument, Zuazua in \cite{zuazua1993exact} proves the first distributed controllability result for the 1d semi-linear wave equation under the assumption \eqref{Assumption_on_f_with_p} with $p=2$. Subsequently, the above result is generalized by the same approach to the multidimensional case with $p=1/2$ in \cite{li2000exact} and then with $p=3/2$ in \cite{fu2019carleman}. In the specific case of mixed boundary conditions with a Neumann control, to our knowledge, there are very few results. A global exact boundary controllability result for \eqref{system_semilinear}, under the assumption that the non-linearity $f$ is continuous with first derivative uniformly bounded, was given in \cite[Chapter $19$]{leugering2005control} in the space $H^1_{(0)}(\Omega)\times L^2(\Omega)$ for the multidimensional case. Another controllability result is proved in \cite[Theorem $4.22$ p. $178$]{coron2007control} with a Schauder fixed-point argument adapted from \cite{zuazua1993exact} assuming that $f$ is at most linear.

 A first consequence of Theorem \ref{Main_result_Obs_inegality} is the following
\begin{theoreme}\label{First_main_result(Schauder)}
Let $T>2$. Assume that $f\in\mathcal{C}^1(\mathbb{R})$ such that
\begin{equation}\tag{$$H1$$}\label{Hyp_H1}
\limsup\limits_{|r|\rightarrow +\infty} \frac{\left| f(r)\right|}{\left| r\right| \ln^2\left| r\right|} \leq \beta.
\end{equation}
If $\beta$ is small enough then the system \eqref{system_semilinear} is exactly controllable in time $T$. 
\end{theoreme}
The proof is based on a Schauder fixed-point argument. In particular, the stability property of the operator results from the observability inequality \eqref{In.obs.g} of Theorem \ref{Main_result_Obs_inegality} and the assumption \eqref{Hyp_H1} on $f$; which can be proved to be optimal in the power of the logarithm by the same argument as in \cite[Section $4$]{zuazua1993exact}.

Remark that Theorem \ref{First_main_result(Schauder)} and all the previous cited results only give the existence of a control function for $(\star\star)$ and is not based on a constructive method. It is only very recently that the question about the  construction of convergent control approximation has emerged: we refer to \cite{munch2022constructive, bottois2023constructive} using a least-squares approach or \cite{cavalcanti2022numerical, bhandari2023exact, claret2023exact} where a Picard iterative scheme is proposed. Under the assumption that the non-linearity is continuous in time and Lipschitz in space with a Lipschitz constant independent of time, we mention the back-and-forth iterations method of \cite{doi:10.1080/00207179.2016.1266513} which is illustrated in the case of the boundary controllability of the Sine-Gordon equation with a mixed boundary conditions.  Eventually, in \cite{doi:10.1142/S0219199706002209}, Coron and Trélat construct a control in a feedback form to move from any steady-state to any other one provided that they are in the same connected component of the set of steady-states.

Assuming an additional growth assumption on $f'$, a second consequence of Theorem \ref{Main_result_Obs_inegality} is the following
\begin{theoreme}\label{Second_main_result}
Let $T>2$. Assume that $f'$ is $\alpha$-Hölder continuous, for some $\alpha\in [0,1]$, and satisfies 
\begin{equation}\tag{$$H2$$}\label{Hyp_H2}
\limsup\limits_{\left|r\right| \rightarrow +\infty} \frac{\left| f'(r)\right|}{\ln^2\left|r\right|} < \beta^\star,
\end{equation}
with $\beta^\star>0$ small enough. In the case $\alpha = 0$, we assume moreover that $\|f'\|_{L^\infty(\mathbb{R})}$ is small enough. There exists a sequence $(y_k, v_k)_{k\in\mathbb{N}}$ which converges strongly, at least with order $1+\alpha$ after a finite number of iterations, to a state-control pair $(y,v)$ of \eqref{system_semilinear}.
\end{theoreme}

As in \cite{munch2022constructive, bottois2023constructive, lemoine2020approximation}, Theorem \ref{Second_main_result} is obtained by a least-squares approach, which consists to minimize the functional 
\begin{equation*}
E(y,v) := \left\|\partial_{tt}y - \partial_{xx}y + f(y) \right\|_{L^2(Q_T)}^2
\end{equation*}
over all the pair $(y,v)\in L^2(Q_T)\times L^2(0,T)$ satisfying $\left(y(\cdot, 0),\partial_ty(\cdot, 0) \right) = (u_0, u_1)$ and $\left( y(\cdot, T),\partial_ty(\cdot, T)\right) = (z_0,z_1)$. In particular, the observability inequality \eqref{In.obs.g} is used in a fundamental way and the assumption $\eqref{Hyp_H2}$ on $f'$ is again a consequence of the expression on the observability constant. 

\vspace{0.3cm}

\textbf{Outline.} The paper is organized as follows. Section \ref{Section_1} is devoted to giving some theoretical results for linear wave equation with potential and source term. In particular, we give the existence, uniqueness and \textit{a priori} estimates for weak solution in the regularity space $H^1_{(0)}(\Omega)\times L^2(\Omega)$ (see Section \ref{Section_1.1}) and for solution in the sense of transposition in the space $L^2(\Omega)\times H^{-1}_{(0)}(\Omega)$ (see Section \ref{Section_1.2}). In Section \ref{Section_1.3} we give the proof of the observability inequality \eqref{In.obs.g} given by Theorem \ref{Main_result_Obs_inegality}. We give then controllability results as an application of \eqref{In.obs.g} in Section \ref{Section3}. Section \ref{Section_3.1} is devoted to controllability of a linear wave equation, while Section \ref{Section_3.2} and Section \ref{Section_3.3} are devoted to Theorem \ref{First_main_result(Schauder)} and Theorem \ref{Second_main_result} respectively. Finally, we present some numerical simulations in Section \ref{Section4} to illustrate our results and, we conclude and give some perspectives in Section \ref{Section5}.

\section{Linear wave equation with mixed boundary condition}\label{Section_1}

\subsection{Existence and uniqueness of solutions in the space $H^1_{(0)}(\Omega)\times L^2(\Omega)$}\label{Section_1.1}

Let $T>2$, $B\in L^1\big(0,T; L^2(\Omega)\big)$, $v\in L^2(0,T)$ and $A\in L^\infty(Q_T)$. We consider the following linear wave equation 
\begin{equation}\label{system_linear}
\left\lbrace 
    	\begin{aligned}
    	&\partial_{tt} y - \partial_{xx} y + Ay = B, &  Q_T, \\ &y(0,\cdot) = 0, \ \ \partial_x y(1,\cdot) = v,   & (0,T) , \\ & \big( y(\cdot,0),\partial_t y(\cdot, 0)\big)  = (u_0, u_1), & \Omega,
    	\end{aligned}
    	\right.
\end{equation}
where $y = y(x,t)$ is the state and $(u_0, u_1)\in H^1_{(0)}(\Omega)\times L^2(\Omega)$ is a given initial data.

We define the solutions of \eqref{system_linear} in a weak sense. 

\begin{definition}\label{Def_weak_sol}
We say that $y\in L^2\big(0,T; H^1_{(0)}(\Omega)\big) \cap H^1\big(0,T; L^2(\Omega)\big)\cap H^2\big(0,T; H^{-1}_{(0)}(\Omega)\big)$ is a \textbf{weak solution} of \eqref{system_linear} if and only if $y$ is solution of the variational formulation 
\begin{equation}\label{FV_weak_solution_y}
\langle \partial_{tt}y(\cdot, t), z \rangle_{H^{-1}_{(0)}(\Omega), H^1_{(0)}(\Omega)} - v(t)z(1) + \int_\Omega \partial_x y \partial_x z \dx + \int_\Omega Ayz \dx = \int_\Omega Bz\dx, \hspace{1cm} \forall z\in H^1_{(0)}(\Omega),
\end{equation}
for a.e. $t\in [0,T]$ and $\left( y(\cdot,0), \partial_t y(\cdot,0)\right) = (u_0, u_1)$ in $\Omega$.
\end{definition}

The well-posedness of system \eqref{system_linear} is proved in the following theorem

\begin{theoreme}\label{Section2_Th1}
Let $A\in L^\infty(Q_T)$, $B\in L^1\left(0,T; L^2(\Omega)\right)$, $v\in L^2(0,T)$ and $(u_0,u_1)\in H^1_{(0)}(\Omega) \times L^2(\Omega)$. There exists a unique weak solution $y$ of \eqref{system_linear} satisfying
\begin{equation*}
y\in\mathcal{C}^0\big([0,T]; H^1_{(0)}(\Omega)\big)\cap \mathcal{C}^1\big([0,T]; L^2(\Omega)\big).
\end{equation*}
Moreover, there exists a constant $C=C(\Omega, T)>0$ such that, for any $t\in [0,T]$,
\begin{equation}\label{A_priori_estimate}
\begin{aligned}
\|y(\cdot, t)\|_{H^1_{(0)}(\Omega)} &+ \|\partial_t y(\cdot, t)\|_{L^2(\Omega)}   \leq Ce^{C\sqrt{\|A\|_{L^\infty(Q_T)}}}\Big( \left\|\left( u_0, u_1\right)\right\|_{H^1_{(0)}(\Omega) \times L^2(\Omega)} + \|B\|_{L^1\left(0,t; L^2(\Omega)\right)} + \|v\|_{L^2(0,t)} \Big).
\end{aligned}
\end{equation}
Finally, $y(1, \cdot)$ belongs to $H^1(0,T)$ and there exists $C=C(\Omega, T)>0$ such that, for any $t\in [0,T]$,
\begin{equation}\label{eq_partial_t_y}
\|\partial_t y(1,\cdot)\|_{L^2(0,t)} \leq Ce^{C\sqrt{\|A\|_{L^\infty(Q_T)}}} \left( \left\|\left( u_0, u_1\right)\right\|_{H^1_{(0)}(\Omega) \times L^2(\Omega)} + \|B\|_{L^1\left(0,t; L^2(\Omega)\right)} + \|v\|_{L^2(0,t)}\right).
\end{equation}
\end{theoreme}
\begin{proof}
The proof is based on well-known arguments. For the existence, as in \cite[Proof of Theorem $8.1$ p. $265$]{lions2012non}, a candidate solution is built using the Faedo-Galerkin method by considering a Hilbert basis of $L^2(\Omega)$ composed of eigenvalues of the Laplacian which is also an orthogonal basis of $H^1_{(0)}(\Omega)$ (see \cite[Théorème IX.$31$ and Remarque $29$ p.$192$-$193$]{brezis1983analyse} for the existence of such a basis). First, we prove that this solution satisfies inequalities \eqref{A_priori_estimate} and \eqref{eq_partial_t_y} and then we show that the above candidate solution is a solution of our problem. Uniqueness is treated separately in a similar way to \cite[Theorem $4$ p.385]{evans2022partial}.
\end{proof}

\subsection{Existence and uniqueness of solutions in the space $L^2(\Omega)\times H^{-1}_{(0)}(\Omega)$}\label{Section_1.2}

This section is dedicated to the existence and uniqueness of solution of the linear wave equation in the space $L^2(\Omega)\times H^{-1}_{(0)}(\Omega)$. In particular, we define the solution of \eqref{system_adjoint} by the transposition method (see \cite[Chapter I, Section $4.2$ p.$47$]{lions2012non}). For any $g\in L^2(Q_T)$, let $w\in \mathcal{C}^0\big( [0,T]; H^1_{(0)}(\Omega)\big)\cap\mathcal{C}^1\big([0,T]; L^2(\Omega)\big)$ be the weak solution of the backward adjoint equation
\begin{equation}\label{Eq_en_w(transpo)}
\left\lbrace 
    	\begin{aligned}
    	&\partial_{tt} w - \partial_{xx} w + Aw = g, &  Q_T, \\ &w(0,\cdot) = 0, \ \ \partial_x w(1,\cdot) = 0,   & (0,T) , \\ & \big( w(\cdot,T),\partial_t w(\cdot, T)\big)  = (0, 0), & \Omega.
    	\end{aligned}
    	\right.
\end{equation}
\begin{definition}\label{def_sol_transpo}
We say that $\varphi\in L^2\big(0,T; L^2(\Omega)\big)$ is a solution in the sense of transposition of \eqref{system_adjoint} if and only if $\varphi$ is solution of the identity
\begin{equation}\label{def_weak_solution}
\begin{aligned}
\int_{Q_T} \varphi(x,t) g(x,t) \dx\dt = \int_{Q_T} B(x,t) w(x,t) \dx\dt &+ \langle \varphi_1, w(\cdot, 0) \rangle_{H^{-1}_{(0)}(\Omega), H^1_{(0)}(\Omega)} \\&  - \int_\Omega \varphi_0(x)\partial_tw(x,0)\dx, \hspace{1cm} \forall g \in L^2(Q_T),
\end{aligned}
\end{equation}
where $w$ is the weak solution of \eqref{Eq_en_w(transpo)}.
\end{definition} 

\begin{theoreme}\label{Theorem_reg_donnee_non_reg}
Let $B=0$, $A\in L^\infty(Q_T)$ and $(\varphi_0, \varphi_1) \in L^2(\Omega) \times H^{-1}_{(0)}(\Omega)$. Then, there exists a unique solution in the sense of transposition of \eqref{system_adjoint} satisfying
\begin{equation*}
\varphi \in \mathcal{C}^0\big([0,T]; L^2(\Omega)\big)\cap \mathcal{C}^1\big([0,T]; H^{-1}_{(0)}(\Omega)\big).
\end{equation*}
Moreover, there exists a constant $C = C(T, \Omega) >0$ such that
\begin{equation}\label{a_priori_estimate_donnee_non_reguliere}
\|\varphi(\cdot, t)\|_{ L^2(\Omega)} + \|\partial_t\varphi(\cdot, t)\|_{H^{-1}_{(0)}(\Omega)} \leq Ce^{C\sqrt{\|A\|_{L^\infty(Q_T)}}}\left\|\left(\varphi_0, \varphi_1\right)\right\|_{L^2(\Omega)\times H^{-1}_{(0)}(\Omega)},
\end{equation}
for any $t\in [0,T]$.
\end{theoreme}
\begin{proof}
Let $(\varphi_0, \varphi_1)\in L^2(\Omega)\times H^{-1}_{(0)}(\Omega)$ and $A\in L^\infty(Q_T)$. We easily check that problem 
\begin{equation}\label{FV_Lax_Migram}
b(\varphi, g) = \ell(g), \hspace{1cm} \forall  g\in L^2(Q_T),
\end{equation}
where
\begin{equation*}
\begin{aligned}
b(\varphi, g) = \int_{Q_T}\varphi g \dx \dt \ \ \ \ \ \text{ and } \ \ \ \ \
 \ell (g) = \langle \varphi_1, w_g(\cdot, 0) \rangle_{H^{-1}_{(0)}(\Omega), H^1_{(0)}(\Omega)} - \int_\Omega \varphi_0(x) \partial_t w_g(x,0)\dx,
\end{aligned}
\end{equation*}
admits a unique solution $\varphi\in L^2(Q_T)$ by the Lax-Milgram theorem (see \cite[Corollaire V.$8$ p.84]{brezis1983analyse}).
\\ Now, let us write
\begin{equation*}
\varphi = \Psi_1 + \Psi_2
\end{equation*}
where $\Psi_1\in \mathcal{C}^0\big([0,T]; L^2(\Omega)\big)\cap\mathcal{C}^1\big((0,T]; H^{-1}_{(0)}(\Omega)\big)$ and $\Psi_2\in\mathcal{C}^0\big([0,T]; H^1_{(0)}(\Omega)\big)\cap \mathcal{C}^1\big([0,T]; L^2(\Omega)\big)$ are respectively the solution in the sense of transposition and the weak solution of
\begin{equation*}\label{Psi_1}
\left\lbrace 
    	\begin{aligned}
    	&\partial_{tt} \Psi_1 - \partial_{xx} \Psi_1 = 0, &  Q_T, \\ &\Psi_1(0,\cdot) = 0, \ \ \partial_x \Psi_1(1,\cdot) = 0,   & (0,T) , \\ & \big( \Psi_1(\cdot,T),\partial_t \Psi_1(\cdot, T)\big)  = (\varphi_0,\varphi_1), & \Omega,
    	\end{aligned}
    	\right. \hspace{1cm}
    	\left\lbrace 
    	\begin{aligned}
    	&\partial_{tt} \Psi_2 - \partial_{xx} \Psi_2 + A\Psi_2 = -A\Psi_1, &  Q_T, \\ &\Psi_2(0,\cdot) = 0, \ \ \partial_x \Psi_2(1,\cdot) = 0,   & (0,T) , \\ & \big( \Psi_2(\cdot,T),\partial_t \Psi_2(\cdot, T)\big)  = (0,0), & \Omega.
    	\end{aligned}
    	\right.
\end{equation*}
Remark that, $\varphi \in \mathcal{C}^0\big([0,T]; L^2(\Omega)\big)\cap\mathcal{C}^1\big((0,T]; H^{-1}_{(0)}(\Omega)\big)$.
Since the energy associated with system in $\Psi_1$ is conserved, we have
\begin{equation}\label{Eq1Psi_1}
\left\|\left(\Psi_1(\cdot, t), \partial_t\Psi_1(\cdot, t)\right)\right\|_{L^2(\Omega)\times H^{-1}_{(0)}(\Omega)}^2 = \left\|\left(\varphi_0, \varphi_1\right)\right\|_{L^2(\Omega)\times H^{-1}_{(0)}(\Omega)}^2, \hspace{1cm} \forall t \in [0,T].
\end{equation}
Moreover, using estimates \eqref{A_priori_estimate} and \eqref{Eq1Psi_1}, we obtain
\begin{equation}\label{Eq1Psi_2}
\begin{aligned}
\left\|\left(\Psi_2(\cdot, t), \partial_t\Psi_2(\cdot, t)\right)\right\|_{H^1_{(0)}(\Omega)\times L^2(\Omega)}^2 \leq &\ Ce^{C\sqrt{\|A\|_{L^\infty(Q_T)}}}\left\|A\Psi_1\right\|_{L^2(Q_T)}^2
\\\leq &\ Ce^{C\sqrt{\|A\|_{L^\infty(Q_T)}}}T\left\|A\right\|_{L^\infty(Q_T)}^2\left\|\left(\varphi_0, \varphi_1\right)\right\|_{L^2(\Omega)\times H^{-1}_{(0)}(\Omega)}^2.
\end{aligned}
\end{equation}
Therefore, since the embeddings $H^1_{(0)}(\Omega)\hookrightarrow L^2(\Omega)$ and $L^2(\Omega)\hookrightarrow H^{-1}_{(0)}(\Omega)$ are continuous, we deduce estimate \eqref{a_priori_estimate_donnee_non_reguliere} using \eqref{Eq1Psi_1} and \eqref{Eq1Psi_2}.
\end{proof}

\begin{remark}
Let $\varphi \in \mathcal{C}^0\big([0,T]; L^2(\Omega)\big)\cap \mathcal{C}^1\big([0,T]; H^{-1}_{(0)}(\Omega)\big)$ the solution in the sense of transposition of \eqref{system_adjoint} associated with $B=0$. As a consequence of the time reversibility of the wave equation, $\varphi$ satisfies 
\begin{equation}\label{Decroissance_exp_energie_adj_bis_bis}
\begin{aligned}
\left\| \left(\varphi_0,\varphi_1\right)\right\|_{L^2(\Omega)\times H^{-1}_{(0)}(\Omega)}  \leq Ce^{C \sqrt{\|A\|_{L^\infty(Q_T)}}}\Big( \|\varphi(\cdot, t)\|_{L^2(\Omega)} + \|\partial_t \varphi(\cdot, t)\|_{H^{-1}_{(0)}(\Omega)} \Big),
\end{aligned}
\end{equation}
for all $t\in [0,T]$.
\end{remark}

\section{Observability estimate}\label{Section_1.3}

This section is devoted to the proof of Theorem \ref{Main_result_Obs_inegality}. It is an adaptation of \cite[Section $3$ p.$120$]{zuazua1993exact} to the boundary case.
Here and in the following, the Dirichlet-Neumann Laplacian inverse operator, noted $\left(-\partial_{xx}\right)^{-1}$, is defined by
\begin{equation*}
\begin{aligned}
\left(-\partial_{xx}\right)^{-1} : H^{-1}_{(0)}(\Omega) & \rightarrow  H^1_{(0)}(\Omega)
\\f &\mapsto u 
\end{aligned}
\end{equation*} 
where $u$ is the unique solution of 
\begin{equation*}
\int_{\Omega} \partial_x u \ \partial_x v \dx = \langle f, v \rangle_{H^{-1}_{(0)}(\Omega), H^1_{(0)}(\Omega)}, \hspace{1cm} \forall v\in H^1_{(0)}(\Omega).
\end{equation*}
In particular, $\left(-\partial_{xx}\right)^{-1}$ is a continuous isomorphism (see \cite[p.$201$]{Lions1988ControlabiliteEP}) such that
\begin{equation*}
\left\|f\right\|_{H^{-1}_{(0)}(\Omega)} = \left\|\left(-\partial_{xx}\right)^{-1}f\right\|_{H^1_{(0)}(\Omega)}, \hspace{1cm} \forall f\in H^{-1}_{(0)}(\Omega).
\end{equation*}

\begin{proposition}{\textbf{Hidden regularity}.}\label{Hidden_regularity}
Let $\varphi\in \mathcal{C}^0\big([0,T]; L^2(\Omega)\big)\cap \mathcal{C}^1\big([0,T]; H^{-1}_{(0)}(\Omega)\big)$ be the solution in the sense of transposition of \eqref{system_adjoint} associated with $A\in L^\infty(Q_T)$, $B=0$ and $(\varphi_0, \varphi_1)\in L^2(\Omega)\times H^{-1}_{(0)}(\Omega)$. Then, $\varphi(1,\cdot)$ belongs to $L^2(0,T)$ and there exists a constant $C=C(T, \Omega)>0$ such that
\begin{equation}\label{Estimation_Hidden_regularity}
\|\varphi(1,\cdot)\|_{L^2(0,T)} \leq Ce^{C\sqrt{\|A\|_{L^\infty(Q_T)}}}\left\|\left(\varphi_0, \varphi_1\right)\right\|_{L^2(\Omega)\times H^{-1}_{(0)}(\Omega)}.
\end{equation}
\end{proposition}
\begin{proof}
Let $(\varphi_0, \varphi_1)\in H^1_{(0)}(\Omega)\times L^2(\Omega)$ and let $\varphi\in \mathcal{C}^0\big([0,T]; H^1_{(0)}(\Omega)\big)\cap \mathcal{C}^1\big([0,T]; L^2(\Omega)\big)\cap \mathcal{C}^2\big([0,T]; H^{-1}_{(0)}(\Omega)\big)$ be the weak solution of \eqref{system_adjoint} associated with $(\varphi_0, \varphi_1)$. We consider $\chi\in H^1_{(0)}(\Omega)$ the unique weak solution of
\begin{equation}
\left\lbrace 
    	\begin{aligned}
    	&-\partial_{xx} \chi = - \varphi_1, &  \Omega, \\ &\chi(0) = 0, \ \partial_x \chi(1) = 0,
    	\end{aligned}
    	\right.
\end{equation}
that is $\chi$ is solution of
\begin{equation*}
\int_\Omega \partial_x\chi\partial_xv \dx = -\langle \varphi_1, v\rangle_{H^{-1}_{(0)}(\Omega), H^1_{(0)}(\Omega)}, \hspace{1cm} \forall v\in H^1_{(0)}(\Omega).
\end{equation*}
We define
\begin{equation*}
 h(x,t) = \int_0^t \varphi(x,s)\ds + \chi(x), \hspace{1cm} \forall (x,t)\in Q_T.
\end{equation*}
In particular, $h \in \mathcal{C}^1\big([0,T]; H^1_{(0)}(\Omega)\big)\cap \mathcal{C}^2\big([0,T]; L^2(\Omega)\big)\cap \mathcal{C}^3\big([0,T]; H^{-1}_{(0)}(\Omega)\big)$.
Since $\varphi$ is a weak solution of system \eqref{system_adjoint}, by Definition \ref{Def_weak_sol}, we have
\begin{equation*}
\langle \partial_{tt}\varphi(\cdot, t), z \rangle_{H^{-1}_{(0)}(\Omega), H^1_{(0)}(\Omega)} + \int_\Omega \partial_x \varphi(x,t) z(x) \dx + \int_\Omega A(x,t)\varphi(x,t) z(x) \dx = 0, \hspace{0.5cm} \forall z\in H^1_{(0)}(\Omega)
\end{equation*}
for any $t\in [0,T]$. We integrate the previous equation on $[0,t], t\in (0,T]$ and we obtain
\begin{equation*}
\begin{aligned}
&\int_0^t \langle \partial_{tt}\varphi(\cdot, s), z \rangle_{H^{-1}_{(0)}(\Omega), H^1_{(0)}(\Omega)} \ds + \int_\Omega\int_0^t \partial_x \varphi(x,s) z(x) \ds\dx + \int_0^t\int_\Omega A(x,s)\varphi(x,s) z(x) \ds\dx = 0
\\ \Leftrightarrow \hspace{0.5cm} & \int_0^t \partial_t\left(\langle \partial_{t}\varphi(\cdot, s), z \rangle_{H^{-1}_{(0)}(\Omega), H^1_{(0)}(\Omega)} \right)\ds + \int_\Omega\int_0^t \partial_x \varphi(x,s) \partial_x z(x) \ds\dx = \int_\Omega \left( -\int_0^t A(x,s)\varphi(x,s)\ds\right) z(x) \dx 
\\ \Leftrightarrow \hspace{0.5cm} & \langle \partial_{t}\varphi(\cdot, t), z \rangle_{H^{-1}_{(0)}(\Omega), H^1_{(0)}(\Omega)} -\langle \partial_{t}\varphi(\cdot, 0), z \rangle_{H^{-1}_{(0)}(\Omega), H^1_{(0)}(\Omega)} + \int_\Omega\int_0^t \partial_x \varphi(x,s) \partial_x z(x) \ds\dx \\& = \int_\Omega \left( -\int_0^t A(x,s)\varphi(x,s)\ds\right) z(x) \dx 
\\ \Leftrightarrow \hspace{0.5cm} & \langle \partial_{t}\varphi(\cdot, t), z \rangle_{H^{-1}_{(0)}(\Omega), H^1_{(0)}(\Omega)} + \int_\Omega \partial_x \chi(x) \partial_x z(x)\dx + \int_\Omega\left(\int_0^t \partial_x \varphi(x,s)\ds\right) \partial_x z(x) \dx \\& = \int_\Omega \left( -\int_0^t A(x,s)\varphi(x,s)\ds\right) z(x) \dx
\\ \Leftrightarrow \hspace{0.5cm} & \langle \partial_{t}\varphi(\cdot, t), z \rangle_{H^{-1}_{(0)}(\Omega), H^1_{(0)}(\Omega)} + \int_\Omega\left(\int_0^t \partial_x \varphi(x,s)\ds + \partial_x \chi(x)\right) \partial_x z(x) \dx = \int_\Omega \left( -\int_0^t A(x,s)\varphi(x,s)\ds\right) z(x) \dx.
\end{aligned}
\end{equation*}
Since $\partial_{tt}h(\cdot,t) = \partial_t\varphi(\cdot,t) \ \ \text{ in } H^{-1}_{(0)}(\Omega)$ and $\partial_xh(\cdot, t) = \int_0^t \partial_x\varphi(\cdot, s)\ds + \partial_x\chi \ \ \text{ in } L^2(\Omega)$, we deduce
\begin{equation*}
\langle \partial_{tt}h(\cdot,t), z \rangle_{H^{-1}_{(0)}(\Omega), H^1_{(0)}(\Omega)} + \int_\Omega\partial_xh(\cdot, t) \partial_x z(x) \dx = \int_\Omega \left( -\int_0^t A(x,s)\varphi(x,s)\ds\right) z(x) \dx
\end{equation*}
i.e. $h$ is a weak solution of
\begin{equation}\label{syst_omega_hidden_regularity}
\left\lbrace 
    	\begin{aligned}
    	&\partial_{tt}h - \partial_{xx}h = - \int_0^t A\varphi\d\sigma, &  Q_T, \\ &h(0,\cdot) = 0, \ \partial_xh(1,\cdot) = 0,  &(0,T)\\
    	& \left( h(\cdot,0), \partial_th(\cdot,0)\right) = (\chi, \varphi_0), & \Omega,
    	\end{aligned}
    	\right.
\end{equation}
with $(\chi, \varphi_0)\in H^1_{(0)}(\Omega)\times H^1_{(0)}(\Omega) \subset H^1_{(0)}(\Omega)\times L^2(\Omega)$.
Thus, using \eqref{eq_partial_t_y}, we get
\begin{equation*}
\|\partial_th(1,\cdot)\|_{L^2(0,T)} \leq C\left( \|\chi\|_{H^1_{(0)}(\Omega)} + \|\varphi_0\|_{L^2(\Omega)} + \left\|\int_0^t A(\cdot, \sigma)\varphi(\cdot, \sigma)\d\sigma \right\|_{L^2(Q_T)}\right).
\end{equation*}
We have (using Cauchy-Schwarz inequality and estimate  \eqref{a_priori_estimate_donnee_non_reguliere})
\begin{equation*}
\begin{aligned}
\left\|\int_0^t A(\cdot, \sigma)\varphi(\cdot, \sigma)\d\sigma \right\|_{L^2(Q_T)}^2 \leq & \ T\|A\|_{L^\infty(Q_T)}^2 \int_0^1 \left| \int_0^t \varphi(x,\sigma) \d\sigma\right|^2 \dx
\\ \leq &\  T\|A\|_{L^\infty(Q_T)}^2 \int_0^1 \left| \sqrt{t}\left(\int_0^t |\varphi(x,\sigma)|^2\d\sigma\right)^{1/2} \right|^2 \dx
\\ \leq &\ T\|A\|_{L^\infty(Q_T)}^2 \int_0^1 T \|\varphi(x,\cdot)\|_{L^2(0,T)}^2\dx
 \leq  T^2\|A\|_{L^\infty(Q_T)}^2\|\varphi\|_{L^2(Q_T)}^2
 \\ \leq &\ T^2C\|A\|_{L^\infty(Q_T)}^2e^{C\sqrt{\|A\|_{L^\infty(Q_T)}}}\left\|\left(\varphi_0, \varphi_1\right)\right\|_{L^2(\Omega)\times H^{-1}_{(0)}(\Omega)}^2 
\end{aligned}
\end{equation*}
and since $\left( - \partial_{xx}\right)^{-1}$ is a continuous isomorphism, $\|\chi\|_{H^1_{(0)}(\Omega)} \leq C \|\varphi_1\|_{H^{-1}_{(0)}(\Omega)}$.
Therefore 
\begin{equation*}\label{Eq1_hidden_regularity}
\|\partial_th(1,\cdot)\|_{L^2(0,T)} \leq Ce^{C\sqrt{\|A\|_{L^\infty(Q_T)}}}\left\|\left(\varphi_0, \varphi_1\right)\right\|_{L^2(\Omega)\times H^{-1}_{(0)}(\Omega)},
\end{equation*}
for all $(\varphi_0, \varphi_1)\in H^1_{(0)}(\Omega)\times L^2(\Omega)$. By a density argument, since $\partial_t h = \varphi$ in $Q_T$, we deduce the result.
\end{proof}

\begin{lemme}\label{prop2}
Let $T>2$ and let $\varphi\in \mathcal{C}^0\big([0,T]; L^2(\Omega)\big)\cap \mathcal{C}^1\big([0,T]; H^{-1}_{(0)}(\Omega)\big)$ be the solution in the sense of transposition of \eqref{system_adjoint} associated with $A\in L^\infty(Q_T)$, $B=0$  and $(\varphi_0, \varphi_1)\in L^2(\Omega)\times H^{-1}_{(0)}(\Omega)$. There exists a constant $C=C(\Omega, T)>0$ such that
\begin{equation}\label{In_prop2}
\int_1^{T-1} \left\|\varphi(\cdot, t)\right\|^2_{L^2(\Omega)} \dt \leq Ce^{C\sqrt{\|A\|_{L^\infty(Q_T)}}} \left\|\varphi(1, \cdot)\right\|^2_{L^2(0,T)}.
\end{equation}
\end{lemme}
\begin{proof}
Let $\psi = \psi(t,x) \in \mathcal{C}^0\big([0,1]; L^2(0,T)\big)\cap \mathcal{C}^1\big([0,1]; H^{-1}_{(0)}(0,T)\big)$ be the solution in the sense of transposition of the following wave equation where the role of the time and space variables has been interchanged:
\begin{equation*}
\left\lbrace 
    	\begin{aligned}
    	&\partial_{xx} \psi - \partial_{tt} \psi - A\psi = 0, &  Q_T, \\ &\psi(0,\cdot) = 0, \ \ \partial_t\psi(T,\cdot) = 0,   & \Omega , \\ & \big( \psi(\cdot,1),\partial_x \psi(\cdot, 1)\big)  = (\varphi(1,\cdot), \partial_x\varphi(1,\cdot)), & (0,T).
    	\end{aligned}
    	\right.
\end{equation*}
The estimate \eqref{a_priori_estimate_donnee_non_reguliere} applied to $\psi$ becomes
\begin{equation*}
\|\psi(\cdot, x)\|_{L^2(0,T)}^2 + \|\partial_x\psi(\cdot, x)\|_{H^{-1}_{(0)}(0,T)}^2 \leq Ce^{C\sqrt{\|A\|_{L^\infty(Q_T)}}} \left\|\big(\psi(\cdot, 1), \partial_x\psi(\cdot, 1)\big)\right\|_{L^2(0,T)\times H^{-1}_{(0)}(0,T)}^2 , \hspace{1cm} \forall x\in \Omega.
\end{equation*}
Since $\psi(t, 1) = \varphi(1,t)$ and $\partial_x\psi(t, 1) = \partial_x\varphi(1,t) = 0$ in $[0,T]$, we obtain
\begin{equation*}
\|\psi(\cdot, x)\|_{L^2(0,T)}^2 + \|\partial_x\psi(\cdot, x)\|_{H^{-1}_{(0)}(0,T)}^2 \leq Ce^{C\sqrt{\|A\|_{L^\infty(Q_T)}}} \left\|\varphi(1, \cdot)\right\|_{L^2(0,T)}^2, \hspace{1cm} \forall x\in \Omega,
\end{equation*}
and we deduce
\begin{equation}\label{m1}
\|\psi\|_{L^2(Q_T)}^2 \leq  C|\Omega|e^{C\sqrt{\|A\|_{L^\infty(Q_T)}}} \|\varphi(1, \cdot)\|_{L^2(0,T)}^2.
\end{equation}
Moreover, since $\varphi = \psi$ in $\tau(1) := \{ (x,t)\in Q_T; \ t\in (1-x, x+T-1)\}\subset Q_T$, we have
\begin{equation}\label{m2}
\begin{aligned}
\|\psi\|_{L^2(Q_T)}^2\geq \int_{\tau(1)}|\psi(t,x)|^2 \dx\dt & = \int_{\tau(1)} |\varphi(x,t)|^2 \dx\dt \\& = \int_0^1 \int_{1-x}^{x+T-1} |\varphi(x,t)|^2 \dx\dt  \geq  \int_0^1 \int_1^{T-1} |\varphi(x,t)|^2 \dx\dt.
\end{aligned}
\end{equation}
Therefore, using \eqref{m1} and \eqref{m2}, we deduce the result.
\end{proof}

\begin{lemme}\label{prop3}
Let $T>2$. There exists $C = C(\Omega, T)>0$ such that the solution $\varphi\in \mathcal{C}^0\big([0,T]; L^2(\Omega)\big)\cap\mathcal{C}^1\big([0,T]; H^{-1}_{(0)}(\Omega)\big)$ in the sense of transposition of \eqref{system_adjoint} associated with $A\in L^\infty(Q_T)$, $B=0$ and $(\varphi_0, \varphi_1)\in L^2(\Omega)\times H^{-1}_{(0)}(\Omega)$ satisfies
\begin{equation}\label{In_prop3}
\int_{t_1}^{t_2}\left\|\partial_t \varphi(\cdot, t)\right\|^2_{H^{-1}_{(0)}(\Omega)} \dt\leq C\left(1+\|A\|_{L^\infty(Q_T)}\right) \int_1^{T-1} \left\| \varphi(\cdot,t)\right\|_{L^2(\Omega)}^2\dt, 
\end{equation}
for any $t_1, t_2\in (1,T-1), t_1<t_2$.
\end{lemme}
\begin{proof} 
Let $(\varphi_0 , \varphi_1) \in \big(H^2(\Omega)\cap H^1_{(0)}(\Omega)\big)\times H^1_{(0)}(\Omega)$ and let $\varphi \in \mathcal{C}^0\big([0,T]; H^2(\Omega)\cap H^1_{(0)}(\Omega)\big)\cap \mathcal{C}^1\big([0,T]; H^1_{(0)}(\Omega)\big)\cap \mathcal{C}^2\big([0,T]; H^{-1}_{(0)}(\Omega)\big)$ be the solution of \eqref{system_adjoint} associated with $(\varphi_0 , \varphi_1) $. Let $t_1, t_2\in (1, T-1)$ with $t_1<t_2$. We take $r\in \mathcal{C}^1_c\big( [1, T-1]\big)$ such that 
\begin{equation*}
\left\lbrace 
    	\begin{aligned}
& r(t) = 1 \text{ for all } t\in [t_1, t_2]\\
& \frac{|r'|^2}{r}\in L^\infty(1, T-1).
\end{aligned}\right.
\end{equation*}
Multiplying the equation in \eqref{system_adjoint} by $r(t)\left(-\partial_{xx}\right)^{-1}\varphi$ and integrating by part on $[1, T-1]$ the term in $\partial_{tt}\varphi$, we obtain
\begin{equation*}
\begin{aligned}
& -\int_{1}^{T-1} r'(t) \int_0^1  \partial_t\varphi(x,t) \left(-\partial_{xx}\right)^{-1}\varphi(x,t) \dx \dt - \int_{1}^{T-1} r(t) \int_0^1  \partial_t\varphi(x,t) \partial_t\left(-\partial_{xx}\right)^{-1}\varphi(x,t) \dx \dt \\& \hspace{0.5cm}
- \int_{1}^{T-1}r(t)\int_0^1  \partial_{xx}\varphi(x,t) \left(-\partial_{xx}\right)^{-1}\varphi(x,t) \dx\dt + \int_{1}^{T-1} \int_0^1 A(x,t)\varphi(x,t) \ r(t)\left(-\partial_{xx}\right)^{-1}\varphi(x,t) \dx\dt  = 0.
\end{aligned}
\end{equation*}
We can check that, for any $t\in [0,T]$
\begin{equation*}
\begin{aligned}
\int_0^1  \partial_t\varphi(x,t) \ \partial_t\left(-\partial_{xx}\right)^{-1}\varphi(x,t) \dx =  \|\partial_t\varphi(\cdot, t)\|_{H^{-1}_{(0)}(\Omega)}^2, \hspace{0.5cm} \int_0^1\partial_{xx}\varphi(x,t) \left(-\partial_{xx}\right)^{-1}\varphi(x,t) \dx = \|\varphi(\cdot, t)\|_{L^2(\Omega)}^2.
\end{aligned}
\end{equation*}
Thus, we obtain
\begin{equation*}
\begin{aligned}
\int_{1}^{T-1}r(t) \ \|\partial_t\varphi(\cdot, t)\|^2_{H^{-1}_{(0)}(\Omega)}\dt =  &\ - \int_{1}^{T-1} r'(t) \int_0^1 \partial_t\varphi(x,t) \left(-\partial_{xx}\right)^{-1} \varphi(x,t) \dx\dt - \int_{1}^{T-1} r(t) \ \|\varphi(\cdot,t)\|^2_{L^2(\Omega)}\dt \\&\ + \int_{1}^{T-1} \int_0^1 A(x,t)\varphi(x,t) \ r(t) \left( -\partial_{xx}\right)^{-1}\varphi(x,t) \dx\dt
\\ \leq &\ \int_{1}^{T-1} \langle \sqrt{r(t})\partial_t\varphi(\cdot,t),  \frac{r'(t)}{\sqrt{r(t)}} \left(-\partial_{xx}\right)^{-1} \varphi(\cdot,t)\rangle_{H^{-1}_{(0)}(\Omega), H^1_{(0)}(\Omega)} \dt \\& \ + \|r\|_{L^\infty(1, T-1)} \int_{1}^{T-1}\|\varphi(\cdot,t)\|^2_{L^2(\Omega)}\dt
\\& + C\left(1+\|r\|_{L^\infty(1,T-1)}\|A\|_{L^\infty(Q_T)}\right) \int_{1}^{T-1} \|\varphi(\cdot, t)\|_{L^2(\Omega)}^2\dt
\\ \leq & \ \frac{1}{2}\int_{1}^{T-1}r(t) \|\partial_t\varphi(\cdot, t)\|_{H^{-1}_{(0)}(\Omega)}^2 \dt \\&+ C\left(1+ \Big\|\frac{|r'(t)|^2}{r(t)}\Big\|_{L^\infty(1,T-1)} +(1+\|A\|_{L^\infty(Q_T)})\|r\|_{L^\infty(1,T-1)} \right)\\ & \hspace{5cm}\times \int_{1}^{T-1} \|\varphi(\cdot, t)\|_{L^2(\Omega)}^2\dt
\end{aligned}
\end{equation*}
i.e. we have \eqref{In_prop3} for all $\varphi$ regular, where $C$ depends on $r$. We conclude the result by a  density argument.
\end{proof}

We are now able to etablish Theorem $\ref{Main_result_Obs_inegality}$.

\begin{proof}[Proof of Theorem $\ref{Main_result_Obs_inegality}$]
Let $\varphi\in \left\{ z\in \mathcal{C}^0\big([0,T]; L^2(\Omega)\big) \cap \mathcal{C}^1\big([0,T]; H^{-1}_{(0)}(\Omega)\big); \ \partial_{tt}\varphi - \partial_{xx}\varphi + A\varphi \in L^2(Q_T) \right\}$. We decompose $\varphi$ as $\varphi = \psi_1 + \psi_2$ where $\psi_1$ and $\psi_2$ are respectively solution in the sense of transposition and weak solution of
\begin{equation*}
\left\lbrace 
    	\begin{aligned}
    	&\partial_{tt} \psi_1 - \partial_{xx} \psi_1 + A\psi_1= 0, &  Q_T, \\ &\psi_1(0,\cdot) = 0, \ \ \partial_x \psi_1(1,\cdot) = 0,   & (0,T) , \\ & \big( \psi_1(\cdot,0),\partial_t \psi_1(\cdot, 0)\big)  = (\varphi_0,\varphi_1), & \Omega,
    	\end{aligned}
    	\right. \hspace{1cm}
    	\left\lbrace 
    	\begin{aligned}
    	&\partial_{tt} \psi_2 - \partial_{xx} \psi_2 + A\psi_2 = \partial_{tt}\varphi - \partial_{xx}\varphi + A\varphi, &  Q_T, \\ &\psi_2(0,\cdot) = 0, \ \ \partial_x \psi_2(1,\cdot) = 0,   & (0,T) , \\ & \big( \psi_2(\cdot,0),\partial_t \psi_2(\cdot, 0)\big)  = (0,0), & \Omega,
    	\end{aligned}
    	\right.
\end{equation*}
Using \eqref{Decroissance_exp_energie_adj_bis_bis} then \eqref{In_prop3}, we have for any $t_1, t_2\in (1,T-1)\subset [0,T], t_1<t_2$
\begin{equation*}
\begin{aligned}
(t_2-t_1)\left\|\left(\varphi_0,\varphi_1\right)\right\|_{L^2(\Omega)\times H^{-1}_{(0)}(\Omega)} &\leq Ce^{C\sqrt{\|A\|_{L^\infty(Q_T)}}}\int_{t_1}^{t_2} \|\psi_1(\cdot, t)\|_{L^2(\Omega)} + \|\partial_t\psi_1(\cdot, t)\|_{H^{-1}_{(0)}(\Omega)} \dt
\\&  \leq Ce^{C\sqrt{\|A\|_{L^\infty(Q_T)}}} \int_{1}^{T-1} \|\psi_1(\cdot, t)\|_{L^2(\Omega)} \dt,
\end{aligned}
\end{equation*}
and using \eqref{In_prop2} and since $\psi_1 = \varphi - \psi_2$, we deduce
\begin{equation*}
\begin{aligned}
\left\|\left(\varphi_0,\varphi_1\right)\right\|_{L^2(\Omega)\times H^{-1}_{(0)}(\Omega)} &\leq \frac{1}{t_1-t_2} Ce^{C\sqrt{\|A\|_{L^\infty(Q_T)}}} \left\|\psi_1(1,\cdot)\right\|_{L^2(0,T)}
\\ &\leq \frac{1}{t_1-t_2} Ce^{C\sqrt{\|A\|_{L^\infty(Q_T)}}} \left(\left\|\varphi(1,\cdot)\right\|_{L^2(0,T)} + \left\|\psi_2(1,\cdot)\right\|_{L^2(0,T)} \right).
\end{aligned}
\end{equation*}
Moreover, since $\psi_2(1,0) = 0$ and using \eqref{eq_partial_t_y}, we have
\begin{equation*}
\begin{aligned}
\left\|\psi_2(1,\cdot)\right\|_{L^2(0,T)} &\leq C\left\|\partial_t \psi_2(1,\cdot) \right\|_{L^2(0,T)} \leq Ce^{C\sqrt{\|A\|_{L^\infty(Q_T)}}} \left\|\partial_{tt}\varphi - \partial_{xx}\varphi + A\varphi\right\|_{L^2(Q_T)},
\end{aligned}
\end{equation*}
and we deduce the result.
\end{proof}

\section{Application to the controllability of a semi-linear wave equation}\label{Section3}

This section is devoted to applications of Theorem \ref{Main_result_Obs_inegality}. In particular, as mentioned in the Introduction, one motivation for inequality \eqref{In.obs.g} is the exact boundary controllability for the semi-linear wave equation \eqref{system_semilinear}.

\subsection{Controllability of a linear wave equation}\label{Section_3.1}

In this section, for any $T>0$, the exact controllability problem associated with \eqref{system_linear} is considered : given $(u_0, u_1), (z_0, z_1)\in H^1_{(0)}(\Omega)\times L^2(\Omega)$, we look for the existence of a control function $v\in L^2(0,T)$ such that the associated solution $y\in \mathcal{C}^0\big([0,T]; H^1_{(0)}(\Omega)\big) \cap \mathcal{C}^1\big([0,T]; L^2(\Omega)\big)$ of \eqref{system_linear} satisfies $\big(y(\cdot, T), \partial_ty(\cdot, t) \big) = (z_0, z_1)$ in $\Omega$.  Moreover, the aim is to get, thanks to estimate \eqref{In.obs.g}, precise estimates of a particular state-control pair in term of the data.

Let $\Phi$ be the space defined by 

\begin{equation*}
\begin{aligned}
\Phi := \Big\{ w\in \mathcal{C}^0\big([0,T]; L^2(\Omega)\big) \cap \mathcal{C}^1\big( [0,T]; H^{-1}_{(0)}(\Omega)\big); &\  w \text{ is the solution of the transposition of } \eqref{system_adjoint}  \\ &  \hspace{4cm} \text{ for some }  B\in L^2(Q_T)\Big\}.
\end{aligned}
\end{equation*}
From \eqref{In.obs.g}, $\Phi$ endowed with the scalar product given by
\begin{equation*}
\left(p,q \right)_{\Phi} := \int_{Q_T} \left( \partial_{tt} p -\partial_{xx}p + Ap\right) \left( \partial_{tt} q -\partial_{xx}q + Aq\right) \dx\dt + \int_0^T p(1,t)q(1,t)\dt, \hspace{1cm} \forall p,q\in \Phi,
\end{equation*}
is a Hilbert space. The main result of this section is a null-controllability result for the linear system \eqref{system_linear}.

\begin{theoreme}\label{Th.8}
Assume $T>2$. For $A\in L^\infty(Q_T)$, $B\in L^2(Q_T)$ and $(u_0, u_1) \in H^1_{(0)}(\Omega)\times L^2(\Omega)$, there exists a unique function $p\in\Phi$ solution of
\begin{equation}\label{FV_en_p}
\left(p,q\right)_\Phi = \int_{Q_T} Bq\dx\dt + \int_\Omega u_1(x) q(x, 0) \dx - \langle\partial_tq(\cdot, 0), u_0\rangle_{H^{-1}_{(0)}(\Omega),H^1_{(0)}(\Omega)}, \hspace{1cm} \forall q\in\Phi.
\end{equation} 
Moreover, $y = \partial_{tt} p -\partial_{xx}p + Ap$ is a controlled trajectory to zero for \eqref{system_linear}, $v = -p(1, \cdot)$ is the associated control function and there exists a constant $C=C(\Omega, T)>0$ such that $(y,v)$ satisfies 
\begin{equation}\label{estimation_norm_L^2_control_BIS}
\|y\|_{L^2(Q_T)} + \|v\|_{L^2(0,T)} \leq Ce^{C\sqrt{\|A\|_{L^\infty(Q_T)}}}\left( \|B\|_{L^2(Q_T)} + \|(u_0, u_1)\|_{H^1_{(0)}(\Omega)\times L^2(\Omega)}  \right).
\end{equation}
\end{theoreme}
\begin{proof}
We refer to \cite[Proposition $2.2$ p.$6$]{CF2013} and \cite[Theorem $6$ p.$7$]{bhandari2023exact} where a similar result is obtained in the case of Dirichlet boundary control.
\end{proof}

\begin{remark}
The state-control pair given by Theorem \ref{Th.8} is the unique solution of the following extremal problem
\begin{equation}\label{Minimize_J}
\left\lbrace 
    	\begin{aligned}
    	& \text{Minimize } \mathcal{J}(y,v) = \|y\|_{L^2(Q_T)}^2 + \|v\|_{L^2(0,T)}^2 
    	\\& \text{Subject to } (y,v) \in C(u_0, u_1;T)
    	\end{aligned}
    	\right.
\end{equation}
where $C(u_0, u_1;T) := \left\{ (y,v)\in L^2(Q_T)\times L^2(0,T); \ (y,v) \text{ is solution of } \eqref{system_linear} \text{ with } y(\cdot, T) = \partial_ty(\cdot, T) = 0 \text{ in } \Omega\right\}$.
\end{remark}

\subsection{Proof of Theorem \ref{First_main_result(Schauder)} by a Schauder fixed-point argument}\label{Section_3.2}

\begin{proof}[Proof of Theorem \ref{First_main_result(Schauder)}]
Let $R>0$. We define the following class
\begin{equation*}
\bar{B}_{\|\cdot\|_\infty}(0,R) := \{ z\in L^\infty(Q_T); \ \|z\|_{L^\infty(Q_T)}\leq R\}.
\end{equation*} 
It suffices to prove that the non-linear operator $K$ defined by
\begin{equation}\label{Def_fixed_point_operator_K1}
\begin{aligned}
K : \bar{B}_{\|\cdot\|_\infty}(0,R) &\rightarrow \bar{B}_{\|\cdot\|_\infty}(0,R)
\\ \xi &\mapsto y
\end{aligned}
\end{equation}
where $y$ is the controlled solution of
\begin{equation}\label{system_linear_schauder}
\left\lbrace 
    	\begin{aligned}
    	&\partial_{tt} y - \partial_{xx} y + \widehat{f}(\xi)y = -f(0), &  Q_T, \\ &y(0,\cdot) = 0, \ \ \partial_x y(1,\cdot) = v,   & (0,T) , \\ & \big( y(\cdot,0),\partial_t y(\cdot, 0)\big)  = (u_0, u_1), & \Omega,
    	\end{aligned}
    	\right. \hspace{1cm} \text{ with } \ \ \widehat{f}(r) = \left\lbrace 
    	\begin{aligned}
    	&\frac{f(r)-f(0)}{r} & \ \ \text{if } r\ne 0
    	\\ &f'(0) & \ \ \text{if } r=0
    	\end{aligned}
    	\right.
\end{equation}
given by Theorem \ref{Th.8}, has a fixed-point for some $R$ large enough. The stability of $K$ is a consequence of the observability inequality \eqref{In.obs.g}. In particular, using \eqref{Hyp_H1}, there exists $\gamma>0$ such that
\begin{equation*}
\|\widehat{f}(\xi)\|_{L^\infty(Q_T)} \leq \gamma + \beta \ln^2( 1+\|\xi\|_{L^\infty(Q_T)}),
\end{equation*}
and we deduce, using \eqref{A_priori_estimate} then \eqref{estimation_norm_L^2_control_BIS},
\begin{equation*}
\begin{aligned}
\|y\|_{L^\infty(Q_T)} & \leq Ce^{C\sqrt{\gamma + \beta \ln^2( 1+\|\xi\|_{L^\infty(Q_T)})}}\left( \|f(0)\|_{L^2(Q_T)} + \|(u_0, u_1)\|_{H^1_{(0)}(\Omega)\times L^2(\Omega)}\right) \\&\leq Ce^{C\sqrt{\gamma}}\left(1+\left\|\xi\right\|_{L^\infty(Q_T)}\right)^{C\sqrt{\beta}}\left( \|f(0)\|_{L^2(Q_T)} + \|(u_0, u_1)\|_{H^1_{(0)}(\Omega)\times L^2(\Omega)}\right).
\end{aligned}
\end{equation*}
From this estimate, we deduce that, if $\beta$ is small enough, there exists $R>0$ such that $\bar{B}_{\|\cdot\|_\infty}(0,R)$ is stable under the map $K$.  The proof for the continuity of $K$ in $\bar{B}_{\|\cdot\|_\infty}(0,R)$ and the fact that $K\left(\bar{B}_{\|\cdot\|_\infty}(0,R)\right)$ is a relatively compact subset of $\bar{B}_{\|\cdot\|_\infty}(0,R)$ is very closed to \cite[Proof of Proposition $2$ p.$92$ and Proposition $3$ p.$93$]{bhandari2023exact}. In particular, the proof is based on the compact embedding (see \cite[Corollary $8$ p.$90$]{simon1986compact})
\begin{equation*}
\mathcal{C}^0\big([0,T]; H^1_{(0)}(\Omega)\big)\cap \mathcal{C}^1\big([0,T]; L^2(\Omega)\big) \hookrightarrow L^\infty(Q_T).
\end{equation*} 
We conclude the existence of a fixed-point for $K$ by the Schauder  theorem.
\end{proof}

\begin{remark}
Following the blow up argument of \cite[Section $4$ p.$124$]{zuazua1993exact}, we can prove that the exponent $2$ in the logarithm in \eqref{Hyp_H1} is optimal.
\end{remark}
 
\begin{remark}
\textit{A priori}, the operator $K$ is not a contraction. In particular, we cannot explicitly construct a control using the Banach-Picard theorem. We refer to Section $\ref{Section_4.2}$ where divergence of the sequence 
\begin{equation}\tag{$$PF1$$}\label{PF1}
y_0 \in L^\infty(Q_T), \hspace{1cm} y_{k+1} = K(y_k), \ k\geq 0.
\end{equation}
is observed numerically.
\end{remark}

\subsection{Proof of Theorem \ref{Second_main_result} by a least-squares approach}\label{Section_3.3}

The motivation of this section is the approximation of exact controls for \eqref{system_semilinear}. Recently, a construction, based on a least-squares approach, of convergent sequence have been initially proposed in \cite{lemoine2020approximation} for the heat equation and have been then adapting for the wave equation in \cite{munch2022constructive} (and also in \cite{bottois2023constructive}) in the internal control case. In particular, this section aims to show that the observability inequality \eqref{In.obs.g} allows to extend the result \cite[Theorem $2$ p.$8$]{munch2022constructive} in our boundary control case.

For any $\alpha\in [0,1]$, we define the space
\begin{equation*}
W_\alpha := \{ f\in \mathcal{C}^1(\mathbb{R}); \  [f']_\alpha < +\infty \}, \hspace{1cm} [f']_\alpha := \sup\limits_{a,b\in\mathbb{R}, a\ne b} \frac{|f'(a)-f'(b)|}{|a-b|^\alpha}.
\end{equation*}
The functional framework is as follows:
\begin{itemize}
\item[$\bullet$] We consider the Hilbert space $\mathcal{H}$ defined by 

\begin{equation*}
\begin{aligned}
\mathcal{H} :=  \ \Bigl\{ (y,v)\in L^2(Q_T)\times L^2(0,T); &\  y\in \mathcal{C}^0\big([0,T]; H^1_{(0)}(\Omega)\big)\cap \mathcal{C}^1\big([0,T]; L^2(\Omega)\big) \text{ is the weak solution } \\& \hspace{5.3cm} \text{ of } \eqref{system_linear} \text{ for some } B\in L^2(Q_T) \Bigr\}
\end{aligned}
\end{equation*}
endowed with the scalar product
\begin{equation*}
\begin{aligned}
\Big((y,v), (\bar{y}, \bar{v}) \Big)_{\mathcal{H}} := &\ \big( \partial_{tt} y - \partial_{xx} y, \partial_{tt} \bar{y} - \partial_{xx} \bar{y} \big)_{L^2(Q_T)} + \big( y, \bar{y}\big)_{L^2(Q_T)} \\&\ + \big(v,\bar{v}\big)_{L^2(0,T)} + \Big( \big(y(\cdot, 0), \partial_ty(\cdot, 0)\big), \big(\bar{y}(\cdot, 0), \partial_t\bar{y}(\cdot, 0)\big) \Big)_{H^1_{(0)}(\Omega)\times L^2(\Omega)}
\end{aligned}
\end{equation*}
and the norm $\|\cdot\|_{\mathcal{H}} := \sqrt{(\cdot, \cdot)_{\mathcal{H}}}$.
\item[$\bullet$] We introduce $\mathcal{A}$ and $\mathcal{A}_0$ the closed subspaces of $\mathcal{H}$ defined by
\begin{equation*}
\begin{aligned}
&\mathcal{A} := \Bigl\{ (y,v)\in \mathcal{H}; \ \big( y(\cdot, 0), \partial_t y(\cdot, 0)\big) = (u_0, u_1), \ \big( y(\cdot, T), \partial_t y(\cdot, T)\big) = (z_0, z_1) \Bigr\},
\\ &\mathcal{A}_0 := \Bigl\{ (y,v)\in \mathcal{H}; \ \big( y(\cdot, 0), \partial_t y(\cdot, 0)\big) = (0, 0), \ \big( y(\cdot, T), \partial_t y(\cdot, T)\big) = (0, 0) \Bigr\}.
\end{aligned}
\end{equation*}
\end{itemize}
We assume \eqref{Hyp_H2}.
In particular, using \eqref{A_priori_estimate} and since $H^1(\Omega) \hookrightarrow \mathcal{C}^0(\overline{\Omega})$, for any $(y,v)\in\mathcal{A}$, $y\in L^\infty(Q_T)$ and thus $f(y) \in L^2(Q_T)$.
We then consider the following non convex (well-defined) extremal problem :
\begin{equation}\label{extremal_problem_E}
\boxed{
\min\limits_{(y,v)\in \mathcal{A}} E(y,v), \hspace{1cm} E(y,v) = \frac{1}{2} \|\partial_{tt}y-\partial_{xx}y + f(y)\|_{L^2(Q_T)}^2.}
\end{equation}
Remark that the infimum of $E$ is reached and is equal to $0$ since any controlled solution of \eqref{system_semilinear}, with its asssociated control, is a zero of $E$. Conversely, any zero of $E$ is a state-control pair of \eqref{system_semilinear}. We have the following property for $E$:
\begin{proposition}\label{P3}
Let $T>2$. For any $(y,v)\in\mathcal{A}$, there exists a constant $C=C(T, \Omega)>0$ such that
\begin{equation}\label{estimateP3}
\sqrt{E(y,v)} \leq  Ce^{C\sqrt{\|f'(y)\|_{L^\infty(Q_T)}}} \|E'(y,f)\|_{\mathcal{A}'_0},
\end{equation}
where $\mathcal{A}_0'$ is the topological dual\footnote{endowed with the norm $\left\|E'(y,v)\right\|_{\mathcal{A}'_0} := \sup\limits_{(Y,V)\in \mathcal{A}_0\backslash \{0\}} \frac{E'(y,v)\cdot(Y,V)}{\left\|(Y,V)\right\|_\mathcal{H}}$.} of $\mathcal{A}_0$.
\end{proposition}
\begin{proof}
For details, we refer to \cite[Proposition $1$, $(iii)$ p.$5$]{munch2022constructive}.
\end{proof}
We deduce that any critical point $(y,v)\in\mathcal{A}$ of $E$ is a zero of $E$. In particular, any minimizing sequence $(y_k, v_k)_{k\in\mathbb{N}}\subset \mathcal{A}$ of $E$ such that $\left\|f'(y_k)\right\|_{L^\infty(Q_T)}$ is uniformly bounded with respect to $k\in\mathbb{N}$ converges to a global minimum of $E$, and thus converges to a state-control pair for \eqref{system_semilinear}.  Remark that, as in \cite[Proposition $1$, $(ii)$ p.$5$]{munch2022constructive}, for any $(y,v)\in \mathcal{A}$ 
\begin{equation*}
E'(y,v)\cdot (Y,V) = 2E(y,v),
\end{equation*}
where $(Y,V)$ is the solution of 
\begin{equation}\label{Syst(Y_k,V_k)}
\left\lbrace 
    	\begin{aligned}
    	&\partial_{tt} Y_k - \partial_{xx} Y_k + f'(y_k)Y_k = \partial_{tt}y_k - \partial_{xx}y_k + f(y_k), &  Q_T, \\ &Y_k(0,\cdot) = 0, \ \partial_xY_k(1,\cdot) = V_k,  & (0,T) , \\ & \Big( Y_k(\cdot,0),\partial_t Y_k(\cdot, 0)\Big)  = (0, 0), & \Omega,
    	\end{aligned}
    	\right.
\end{equation}
associated with $(y,v)$. Thus, $-(Y,V)$ is a descent direction for $E$.
This leads us to consider the sequence $(y_k,v_k)_{k\in\mathbb{N}}$ in $\mathcal{A}$ defined by
\begin{equation}\tag{$$LS$$}\label{LS_algo}
\left\lbrace 
    	\begin{aligned}
    	& (y_0, v_0)\in\mathcal{A} \\ & (y_{k+1}, v_{k+1}) = (y_k, v_k)-\lambda_k (Y_k, V_k) \\ & \lambda_k = \mathrm{argmin}_{\lambda\in [0,1]} E\left((y_k, v_k) - \lambda(Y_k, V_k)\right)
    	\end{aligned}
    	\right.
\end{equation}
where $(Y_k, V_k)\in\mathcal{A}_0$ is the solution of \eqref{Syst(Y_k,V_k)} satisfying the extremal problem \eqref{Minimize_J}.

The main result of this section is a convergence of the least-squares algorithm $($LS$)$:

\begin{theoreme}\label{Th_Convergence_results}
Assume that $f\in W_\alpha$, for some $\alpha\in (0,1]$, and that $f'$ satisfies \eqref{Hyp_H2} with $\beta^\star$ small enough.
Let $(y_k,v_k)_{k\in\mathbb{N}}$ be the sequence defined by $($LS$)$. Then, 
\begin{itemize}
\item The sequence $\big(E(y_k,v_k)\big)_{k\in\mathbb{N}}$ decays to zero as $k\rightarrow \infty$.
\item The sequence $(\lambda_k)_{k\in\mathbb{N}}$ converges to $1$ as $k\rightarrow \infty$.
\item The sequence $(y_k, v_k)_{k\in\mathbb{N}}$ strongly converges to a state-control pair $(\bar{y}, \bar{v})$ of \eqref{system_semilinear}.
\end{itemize}
Moreover, the convergence of all these sequences is first at least linear and then at least with order $1+\alpha$ after a finite number of iterations.
\\ In the case $\alpha = 0$, the result is still true if we assume moreover that $\|f'\|_{L^\infty(\mathbb{R})}$ is small enough.
\end{theoreme}
\begin{proof}
The calculations differ only slightly from \cite[Section $3$]{munch2022constructive}. In the same way, the key point in the proof is the uniformly bounded character of the sequence $\left(\|y_k\|_{L^\infty(Q_T)}\right)_{k\in\mathbb{N}}$. The above property can be proved by induction using \eqref{Hyp_H2} (see \cite[Proof of Theorem $2$ p.$13$]{bottois2023constructive}) and allows us to keep a uniform bound of the sequence of observability constants $\left(Ce^{C\sqrt{\|f'(y_k)\|_{L^\infty(Q_T)}}} \right)_{k\in\mathbb{N}}$ appearing in particular in Proposition \ref{P3}.
\end{proof}

\begin{remark}
We can remove the assumption \eqref{Hyp_H2}, leading to a local controllability result:  If $\|(u_0, u_1)\|_{H^1_{(0)}(\Omega)\times L^2(\Omega)}$ is small enough then the sequence $(y_k, v_k)_{k\in\mathbb{N}}$ defined by $($LS$)$ strongly converges to a state-control pair for  \eqref{system_semilinear}.
For details, we refer to \cite[Proposition $5$ p.$15$]{bottois2023constructive}.
\end{remark}
%
%

\section{Numerical simulations}\label{Section4}

\subsection{Least-squares algorithm $($LS$)$}

In order to consider a control vanishing in time $t=0$ and $t=T$, we introduce a cut-off function $\eta\in \mathcal{C}^1_c([0,T])$ and then we consider the state-control pairs in $L^2(Q_T)\times L^2_\eta(0,T)$ where $L^2_\eta(0,T):= \left\{v; \left\|\eta^{-1}v\right\|_{L^2(0,T)} < +\infty \right\}$. The least-squares algorithm is therefore given in Algorithm \ref{Algo_LS}. 

\begin{algorithm}
\caption{Least-squares algorithm}\label{Algo_LS}
\textbf{Initialization}
\begin{itemize}
\item Compute $(y_0, v_0)\in\mathcal{A}$ solution of
\begin{equation}\label{Simu_system_(y0,v0)}
\left\lbrace 
    	\begin{aligned}
    	&\partial_{tt} y_0 - \partial_{xx} y_0 = 0, &  Q_T, \\ &y_0(0,\cdot) = 0,\ \  \partial_xy_0(1,\cdot) = v_0, & (0,T) , \\ & \big( y_0(\cdot,0),\partial_t y_0(\cdot, 0)\big)  = (u_0, u_1), & \Omega,
    	\\ & \big( y_0(\cdot,T),\partial_t y_0(\cdot, T)\big)  = (0, 0), & \Omega,
    	\end{aligned}
    	\right.
\end{equation}
where $v_0$ minimizes the functional 
\begin{equation*}
\mathcal{J}(y,v) := \left( \|y\|^2_{L^2(Q_T)} + \|v\|_{L^2_\eta(0,T)}^2\right).
\end{equation*}
\item Compute $E(y_0, v_0) = \frac{1}{2}\|\partial_{tt}y_0 - \partial_{xx} y_0 + f(y_0)\|_{L^2(Q_T)}^2$.
\end{itemize}
\While{$\sqrt{2E(y_k,v_k)}>10^{-5}$}{
\begin{itemize}
\item Compute the optimal direction $(Y_k^1, V_k^1)\in\mathcal{A}_0$
solution of 
\begin{equation}\label{Simu_system_(Y_n^1, V_n^1)}
\left\lbrace 
    	\begin{aligned}
    	&\partial_{tt} Y_k^1 - \partial_{xx} Y_k^1 + f'(y_k)Y_k^1 = \partial_{tt}y_k - \partial_{xx} y_k + f(y_k), &  Q_T, \\ &Y_k^1(0,\cdot) = 0,\ \  \partial_xY_k^1(1,\cdot) = V_k^1, & (0,T), \\ & \big( Y_k^1(\cdot,0),\partial_t Y_k^1(\cdot, 0)\big)  = (0, 0), & \Omega,
    	\\ & \big( Y_k^1(\cdot,T),\partial_t Y_k^1(\cdot, T)\big)  = (0, 0), & \Omega,
    	\end{aligned}
    	\right.
\end{equation}
where $V_k^1$ minimizes the functional $\mathcal{J}$.
\item Compute the optimal descent step\footnotemark
\begin{equation*}
\begin{aligned}
\lambda_k &= \mathrm{argmin}_{\lambda\in [0,1]} E\left(y_k - \lambda Y_k^1, v_k - \lambda V_k^1\right) \\&= \mathrm{argmin}_{\lambda\in [0,1]}\|(1-\lambda)\left(\partial_{tt}y_k - \partial_{xx}y_k + f(y_k)\right) + G_k(\lambda)\|_{L^2(Q_T)}^2
\end{aligned}
\end{equation*}
where $G_k(\lambda)=f(y_k - \lambda Y_k^1) - f(y_k) + \lambda f'(y_k)Y_k^1$.
\item Update $(y_{k+1}, v_{k+1}) = \left( y_k - \lambda_k Y_k^1, v_k - \lambda_k V_k^1 \right)$
\item Compute $E(y_{k+1}, v_{k+1}) = \frac{1}{2}\|\partial_{tt}y_{k+1} - \partial_{xx}y_{k+1} + f(y_{k+1})\|_{L^2(Q_T)}^2$.
\item Do $k = k+1$
\end{itemize}
}
\end{algorithm}
\footnotetext{by the trichotomy method}

\begin{lemme}\label{Lem_sim_num}
Let $(y_0, v_0)$ the unique controlled solution of \eqref{system_semilinear} with $f=0$ minimizing $\mathcal{J}$ defined by \eqref{Minimize_J}. Then, $E(y_{k+1}, v_{k+1})$ is expressed explicitly in terms of $f(y_{k'})$ and $f'(y_{k'})$, $0 \leq k'\leq k$:
\begin{equation}\label{Formule_E}
E(y_{k+1}, v_{k+1}) = \frac{1}{2}\left\|\left(\prod\limits_{i=0}^k(1-\lambda_i)\right)f(y_0) + \left(\sum\limits_{j=0}^{k-1} G_j(\lambda_j) \prod\limits_{i=j+1}^k (1-\lambda_i)\right)+ G_k(\lambda_k)\right\|_{L^2(Q_T)}^2
\end{equation}
where
\begin{equation*}
G_k(\lambda) = f(y_k-\lambda Y_k^1) - f(y_k) + \lambda f'(y_k)Y_k^1.
\end{equation*}
\end{lemme}
 
\begin{proof}
By induction, let us prove that
\begin{equation}\label{Formule_E1}
\partial_{tt}y_{k+1}-\partial_{xx}y_{k+1} + f(y_{k+1}) = \left(\prod\limits_{i=0}^k(1-\lambda_i)\right)f(y_0) + \left(\sum\limits_{j=0}^{k-1} G_j(\lambda_j) \prod\limits_{i=j+1}^k (1-\lambda_i)\right)+ G_k(\lambda_k).
\end{equation}
For $k=0$, we have
\begin{equation*}
\begin{aligned}
\partial_{tt}y_1-\partial_{xx}y_1 + f(y_1) 
&=  \partial_{tt}\left(y_0 - \lambda_0 Y_0\right)-\partial_{xx}\left(y_0 - \lambda_0 Y_0\right) + f\left(y_0 - \lambda_0 Y_0\right) 
\\ & = \partial_{tt}y_0 - \partial_{xx}y_0 + f(y_0) - \lambda_0\left(\partial_{tt}Y_0 - \partial_{xx}Y_0 + f'(y_0)Y_0 \right) + G_0(\lambda_0)
\\ & = f(y_0) - \lambda_0\big(\partial_{tt}y_0 - \partial_{xx}y_0 + f(y_0) \big) + G_0(\lambda_0)
\\ & = (1-\lambda_0)f(y_0) + G_0(\lambda_0).
\end{aligned}
\end{equation*}
Assume \eqref{Formule_E} for some $k\in\mathbb{N}$. Then,
\begin{equation*}
\begin{aligned}
\partial_{tt}y_{k+1}&-\partial_{xx}y_{k+1} + f(y_{k+1}) 
= \partial_{tt}\left(y_{k}- \lambda_kY_k\right)-\partial_{xx}\left(y_{k}- \lambda_kY_k\right) + f\big(y_{k}- \lambda_kY_k\big) 
\\ &= \partial_{tt}y_{k}-\partial_{xx}y_{k} + f(y_{k}) - \lambda_k\big(\partial_{tt}Y_k - \partial_{xx}Y_k + f'(y_k)Y_k \big) + G_k(\lambda_k) 
\\ &= \partial_{tt}y_{k}-\partial_{xx}y_{k} + f(y_{k}) - \lambda_k\big(\partial_{tt}y_{k}-\partial_{xx}y_{k} + f(y_{k}) \big) + G_k(\lambda_k) 
\\ &= (1-\lambda_k)\left(\partial_{tt}y_{k}-\partial_{xx}y_{k} + f(y_{k})\right)  + G_k(\lambda_k) 
\\& =(1-\lambda_k)\left(\left(\prod\limits_{i=0}^{k-1}(1-\lambda_i)\right)f(y_0) + \left(\sum\limits_{j=0}^{k-2} G_j(\lambda_j) \prod\limits_{i=j+1}^{k-1} (1-\lambda_i)\right)+ G_{k-1}(\lambda_{k-1})\right) + G_k(\lambda_k) 
\\& = \left(\prod\limits_{i=0}^{k}(1-\lambda_i)\right)f(y_0) + \left(\sum\limits_{j=0}^{k-2} G_j(\lambda_j) \prod\limits_{i=j+1}^{k} (1-\lambda_i)\right)+ (1-\lambda_k)G_{k-1}(\lambda_{k-1}) + G_k(\lambda_k) 
\\& = \left(\prod\limits_{i=0}^{k}(1-\lambda_i)\right)f(y_0) + \left(\sum\limits_{j=0}^{k-1} G_j(\lambda_j) \prod\limits_{i=j+1}^{k} (1-\lambda_i)\right) + G_k(\lambda_k).
\end{aligned}
\end{equation*}
Therefore \eqref{Formule_E1} is true for any $n\in\mathbb{N}$. Since $E(y_k, v_k) = \left\| \partial_{tt} y_k - \partial_{xx} y_k + f(y_k) \right\|_{L^2(Q_T)}^2$ the result follows.
\end{proof}

\begin{remark}
Lemma \ref{Lem_sim_num} avoids second differentiation in time and space in the evaluation of $E(y_{k+1}, v_{k+1})$, for all $k\in\mathbb{N}$.
\end{remark}

\subsection{Discretization by a conformal space-time finite elements method}\label{Section_4.2_1}

To compute $(y_0, v_0)$ and $(Y_k, V_k)$ for any $k\in\mathbb{N}$, we discretize the variational formulation \eqref{FV_en_p} by using a conformal space-time finite elements method, we refer to \cite[Section $5.1$]{bhandari2023exact} and \cite{cindea2015mixed}. We introduce a triangulation $\mathrm{T_h}$ of $Q_T$ such that $\overline{Q_T} = \cup_{K\in \mathrm{T_h}} K$. We assume that $\{\mathrm{T_h}\}_{h>0}$ is a regular family. We approximate the variable $p$ by the variable $p_h$ in the finite-dimensional space
\begin{equation*}
  P_h := \left\{ p_h\in \mathcal{C}^1\left(\overline{Q_T}\right); \ {p_h}_{|K}\in \mathbb{P}(K) \text{ for all } k\in \mathcal{T}_h \right\} \subset \Phi
\end{equation*}
where $\mathbb{P}(K)$ denotes the reduced Hsieh-Clough-Tocher (HCT) $\mathcal{C}^1$-finite element (see \cite[p. $340$]{ciarlet2002finite})
and the controlled solution $y$ are approximated by $y_h$ in the  finite-dimensional space
\begin{equation*}
Q_h := \left\{ q_h\in \mathcal{C}^0\left(\overline{Q_T}\right); \ {q_h}_{|K}\in \mathbb{Q}(K) \text{ for all } k\in \mathcal{T}_h \right\} \subset L^2(Q_T)
\end{equation*}
where $\mathbb{Q}(K)$ denotes the space of polynomials of degree one.

\subsection{Experiments}\label{Section_4.2}

We use a regular space-time mesh composed of $60000$ triangles corresponding to the discretization parameter $h \approx 1.414\times 10^{-2}$ and we denote by 
\begin{equation*}
k^\star := \min\limits_{k\geq 0}\left\{ \sqrt{2E(y_k, v_k)} > 10^{-5} \right\} +1
\end{equation*}
the number of iterations obtained when the least-squares algorithm stops. We set the controllability time equal to $T=3$, the final data to $(z_0, z_1)=(0,0)$ and we define the following non-linear function $f$
\begin{equation}\label{NL_f}
f(r) = c_f r\ln^{2}(1+|r|), \hspace{1cm} \forall r\in\mathbb{R},
\end{equation}
for $c_f\in \mathbb{R}$. In particular, $f\in W_\alpha$ with $\alpha=1$ and $f'$ satisfies \eqref{Hyp_H2}.
Note that the unfavorable case in which the norm of the corresponding uncontrolled solution of \eqref{system_semilinear} grows corresponds to negative values of the parameter $c_f$.  Finally, we define the following cut-off function
\begin{equation*}
\eta(t) = \frac{e^{-\frac{1}{2(t+10^{-6})}}e^{-\frac{1}{2(T-t+10^{-6})}}}{e^{-\frac{1}{T+10^{-6}}}e^{-\frac{1}{T+10^{-6}}}}.
\end{equation*}

\paragraph{I. Experiments with the initial state $u_0(x) = c_{u_0}\left(\cos(\pi x)-1\right)$, $c_{u_0}\in \mathbb{R}$, and $u_1(x) = 0$ in $\Omega$.}

\subparagraph{I-a. We fix $(c_f, c_{u_0})$.} We compute the sequence $(y_k, v_k)_{k\in\mathbb{N}}$ given by the least-squares algorithm associated with the fixed parameters $c_f = -1$ and $c_{u_0} = 20$. The convergence of the algorithm is observed after $k^\star = 13$ iterations.  Table \ref{Table1} collects some numerical values with respect to the iterations $k$. Figure \ref{Fig1} represents the evolution of the error $\sqrt{E(y_k, v_k)}$ as well as the optimal steps $\lambda_k$ with respect to  $k$. As expected, the sequence $(\lambda_k)_{k\in\mathbb{N}^\star}$ converges to $1$ and we observe the change in the convergence speed after a finite number of iterations: first, the optimal step is close to zero and the error decreases linearly and then, after $10$ iterations, the optimal step reaches $1$ while the error decreases quadratically. Finally, the controlled solution $y_{k^\star}$ obtained is shown in Figure \ref{Fig2} as well as its associated control $v_{k^\star}$.

\begin{table}[!h]
\center
\begin{small}
\begin{tabular}{|c|c|c|c|c|c|c|c|c|c|}
  \hline
  $k$ & $\sqrt{2E(y_k, v_k)}$ & $\lambda_k$ & $\|y_{k}\|_{L^2(Q_T)}$ & $\|v_{k}\|_{L^2(0,T)}$ & $\frac{\|y_{k} - y_{k-1}\|_{L^2(Q_T)}}{\|y_{k-1}\|_{L^2(Q_T)}}$ & $\frac{\|v_{k} - v_{k -1}\|_{L^2(0,T)}}{\|v_{k - 1}\|_{L^2(0,T)}}$ \\
  \hline
   $0$& $1.915\times 10^{2}$  & $-$ & $1.882\times 10^1$ & $2.873\times 10^1$ & $-$ & $-$  \\
   $1$& $1.798\times 10^{2}$ & $1.067\times 10^{-1}$ & $2.125\times 10^1$ & $2.471\times 10^1$ & $4.926\times 10^{-1}$ & $1.491$  \\
   $2$& $1.617\times 10^{2}$ & $1.068\times 10^{-1}$ & $2.714\times 10^1$ & $6.360\times 10^1$ & $5.618\times 10^{-1}$ & $1.932$  \\
   $3$& $1.474\times 10^{2}$ & $1.497\times 10^{-1}$ & $3.476\times 10^1$ & $1.219\times 10^2$ & $4.566\times 10^{-1}$ & $1.004$ \\
   $4$& $1.320\times 10^{2}$ & $1.481\times 10^{-1}$ & $4.296\times 10^1$ & $1.755\times 10^2$ & $3.525\times 10^{-1}$ & $4.936\times 10^{-1}$ \\
   $5$& $1.195\times 10^2$ & $1.480\times 10^{-1}$ & $5.177\times 10^1$ & $2.173\times 10^2$ & $2.806\times 10^{-1}$ & $3.298\times 10^{-1}$ \\
   $6$& $1.077\times 10^2$ & $1.658\times 10^{-1}$ & $6.125\times 10^1$ & $2.573\times 10^2$ & $2.357\times 10^{-1}$ & $3.005\times 10^{-1}$ \\
   $7$& $9.502\times 10^1$ & $2.039\times 10^{-1}$ & $7.161\times 10^1$ & $3.043\times 10^2$ & $2.082\times 10^{-1}$ & $3.048\times 10^{-1}$ \\
   $8$& $8.030\times 10^1$ & $2.721\times 10^{-1}$ & $8.300\times 10^1$ & $3.659\times 10^2$ & $1.900\times 10^{-1}$ & $3.129\times 10^{-1}$ \\
   $9$& $6.191\times 10^1$ & $4.013\times 10^{-1}$ & $9.547\times 10^1$ & $4.473\times 10^2$ & $1.764\times 10^{-1}$ & $3.136\times 10^{-1}$ \\
   $10$& $3.638\times 10^1$ & $6.966\times 10^{-1}$ & $1.087\times 10^2$ & $5.489\times 10^2$ & $1.614\times 10^{-1}$ & $2.977\times 10^{-1}$ \\
   $11$& $4.818$ & $1$ & $1.145\times 10^2$ & $6.001\times 10^2$ & $6.512\times 10^{-2}$ & $1.235\times 10^{-1}$ \\
   $12$& $3.146\times 10^{-3}$ & $1$ & $1.145\times 10^2$ & $6.003\times 10^2$ & $1.156\times 10^{-3}$ & $2.428\times 10^{-3}$ \\
   $13$& $2.409\times 10^{-8}$ & $1$ & $1.145\times 10^2$ & $6.003\times 10^2$ & $2.826\times 10^{-7}$ & $6.473\times 10^{-7}$ \\
\hline
\end{tabular}
\end{small}
\caption{$c_{u_0} = 20$ and $c_f = -1$. Some norms with respect to the iterations $k$.}\label{Table1}
\end{table}

\begin{figure}[h!]
\centering
\includegraphics[scale=0.48]{"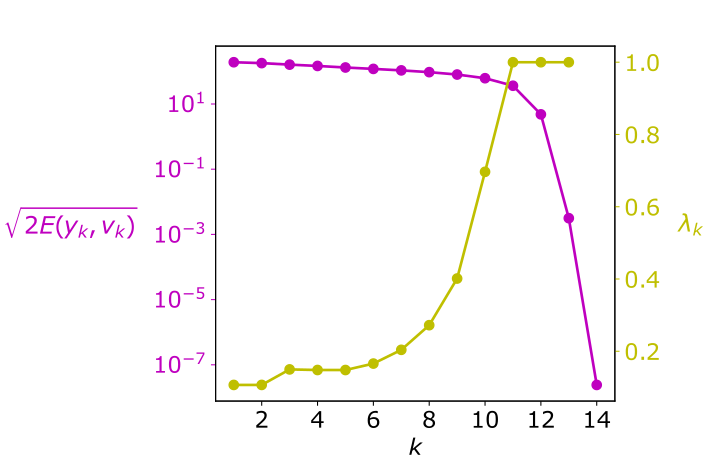"}
\caption{$c_{u_0} = 20$ and $c_f = -1$. Evolution of $\sqrt{2E(y_k, v_k)}$ and $\lambda_k$ with respect to the iterations $k$.} 
\label{Fig1}
\end{figure}

\begin{figure}[h!]
\centering
\includegraphics[scale=0.24]{"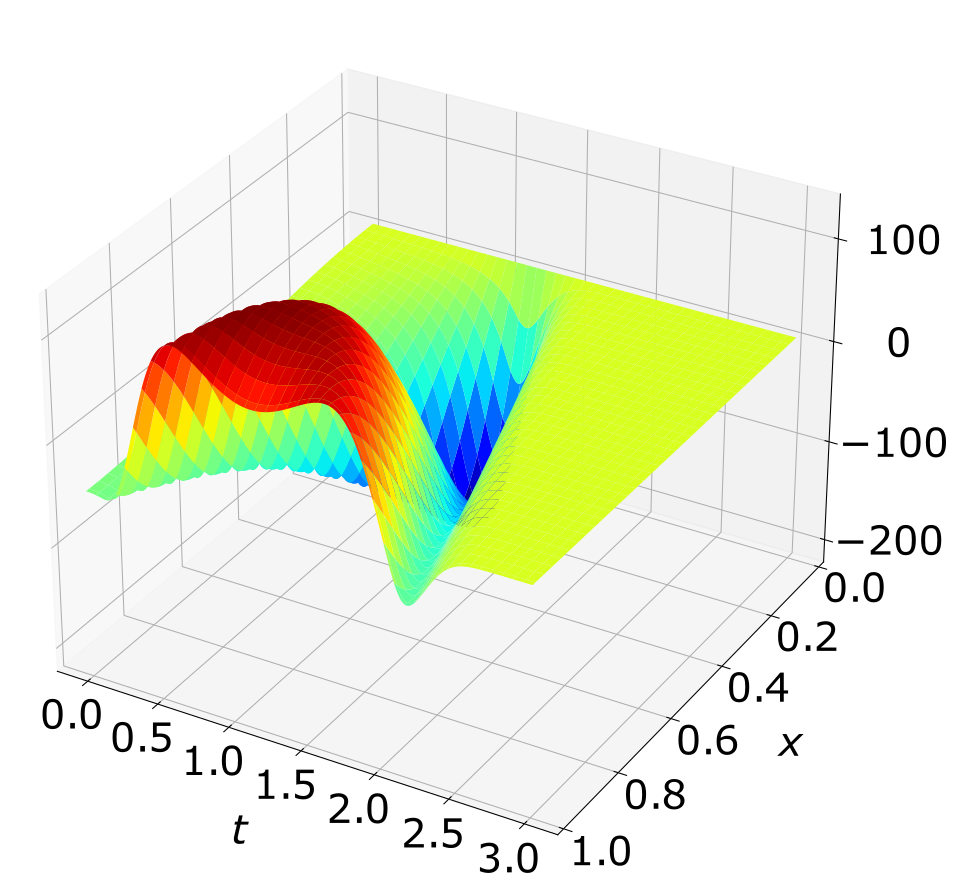"}
\hspace{0.5cm}\includegraphics[scale=0.48]{"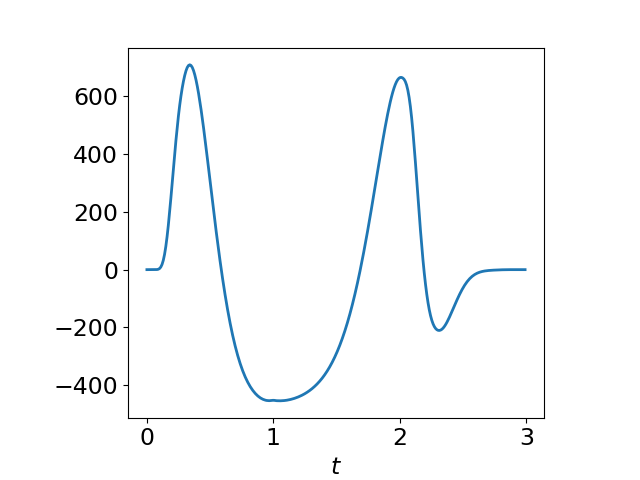"}\caption{$c_{u_0}= 20$ and $c_f = -1$. \textbf{Left:} Representation of the semi-linear controlled solution $y_{k^\star}$ in the space-time domain $Q_T$. \textbf{Right:} Representation of the control $v_{k^\star}$ on $(0,T)$.}
\label{Fig2}
\end{figure}

\subparagraph{I-b. We fix $c_{u_0}$ and we consider several values for $c_f$.} We fix $c_{u_0} = 20$ and we compute $(y_k, v_k)_{k\in\mathbb{N}}$ for several values of $c_f$. Table \ref{ABC} collects the results. Remark that the algorithm fails to converge when $c_f$ is large which is in agreement with our theoretical result that the constant $c_f$ should be small enough.  As expected, the $\|\cdot\|_{L^2(Q_T)}$-norm of $y_{k^\star}$ increases (and thus also the $\|\cdot\|_{L^2(0,T)}$-norm of $v_{k^\star}$) with the absolute value of $c_f$. In particular, due to the non-linearity, for a given $|c_f|$, much more iterations are required in the unfavorable case ($c_f<0$) than in the favorable case ($c_f>0$).

\begin{table}[!ht]
\center
\begin{small}
\begin{tabular}{|c|c|c|c|c|c|c|c|c|c|}
  \hline
  $c_f$ & $\sqrt{2E(y_{k^\star}, v_{k^\star})}$ & $\|y_{k^\star}\|_{L^2(Q_T)}$ & $\|v_{k^\star}\|_{L^2(0,T)}$ & $\frac{\|y_{k^\star} - y_{k^\star-1}\|_{L^2(Q_T)}}{\|y_{k^\star-1}\|_{L^2(Q_T)}}$ & $\frac{\|v_{k^\star} - v_{k^\star -1}\|_{L^2(0,T)}}{\|v_{k^\star - 1}\|_{L^2(0,T)}}$ & $k^\star$ \\
  \hline
    $10$ & $5.461\times 10^{-9}$ & $3.635\times 10^1$ & $1.547\times 10^3$ & $4.805\times 10^{-7}$ & $6.015\times 10^{-7}$ & $27$ \\
   $5$ & $1.726\times 10^{-9}$ & $2.478\times 10^1$ & $3.449\times 10^2$ & $9.589\times 10^{-8}$ & $1.654\times 10^{-8}$ & $9$ \\
   $2$& $5.754\times 10^{-10}$ & $2.025\times 10^1$ & $8.554\times 10^1$ & $4.392\times 10^{-8}$ & $2.470\times 10^{-7}$ & $7$\\
   $1$ & $2.512\times 10^{-8}$ & $2.099\times 10^1$ & $5.895\times 10^1$ & $5.857\times 10^{-6}$ & $2.399\times 10^{-5}$ & $5$ \\
   $-0.5$& $3.656\times 10^{-6}$ & $4.195\times 10^1$ & $1.355\times 10^2$ & $9.689\times 10^{-5}$ & $2.928\times 10^{-4}$ & $5$\\
  $-1$ & $2.409\times 10^{-8}$ & $1.145\times 10^{2}$ & $6.003\times 10^{2}$ & $2.826\times 10^{-7}$ & $6.473\times 10^{-7}$ & $13$\\
   $-1.5$ & $3.314\times 10^{-8}$ & $3.332\times 10^2$ & $2.541\times 10^{3}$ & $4.538\times 10^{-8}$ & $1.635\times 10^{-7}$ & $40$ \\
   $-2$ & $1.217\times 10^{-9}$ & $9.982\times 10^2$ & $1.110\times 10^4$ & $5.914\times 10^{-10}$ & $1.408\times 10^{-9}$ & $143$\\    
  \hline
\end{tabular}
\end{small}
\caption{$c_{u_0} = 20$. Some norms with respect to the parameter $c_f$.}\label{ABC}
\end{table}

\subparagraph{I-c. We fix $c_f$ and we consider several values of $c_{u_0}$.} In this case, we fix $c_{f} = -1$ and we compute $(y_k, v_k)_{k\in\mathbb{N}}$ for several values of $c_u$. Table \ref{Table3} collects the results. We observe that the norm of the control and the controlled solution increase with $|c_{u_0}|$. Moreover, as expected, the algorithm converges even for large values of $c_{u_0}$.
\vspace{0.3cm}

\begin{table}[!h]\label{Table3}
\center
\begin{small}
\begin{tabular}{|c|c|c|c|c|c|c|c|c|c|}
  \hline
  $c_{u_0}$ & $\sqrt{2E(y_{k^\star}, v_{k^\star})}$ & $\|y_{k^\star}\|_{L^2(Q_T)}$ & $\|v_{k^\star}\|_{L^2(0,T)}$ & $\frac{\|y_{k^\star} - y_{k^\star-1}\|_{L^2(Q_T)}}{\|y_{k^\star-1}\|_{L^2(Q_T)}}$ & $\frac{\|v_{k^\star} - v_{k^\star -1}\|_{L^2(0,T)}}{\|v_{k^\star - 1}\|_{L^2(0,T)}}$ & $k^\star$ \\
  \hline
   $1$ & $7.081\times 10^{-6}$ & $1.005$ & $1.326$ & $3.582\times 10^{-3}$ & $1.150\times 10^{-2}$ & $2$\\
   $50$ & $6.893\times 10^{-10}$ & $8.640\times 10^2$ & $6.700\times 10^3$ & $4.838\times 10^{-10}$ & $1.670\times 10^{-9}$ & $31$\\
   $100$ & $9.124\times 10^{-11}$ & $4.128\times 10^{3}$ & $4.399\times 10^4$ & $2.358\times 10^{-11}$ & $4.319\times 10^{-11}$ & $62$\\
   $500$ & $1.865\times 10^{-8}$ & $1.891\times 10^5$ & $2.974\times 10^6$ & $8.448\times 10^{-11}$ & $1.034\times 10^{-10}$ & $444$\\
  \hline
\end{tabular}
\end{small}
\caption{$c_{f} = -1$. Some norms with respect to the parameter $c_{u_0}$.}\label{Table3}
\end{table}

\subparagraph{I-d. Influence of the non-linearity $f$.}
In this case, we fix $c_{u_0} = 1$. For $c_f\in \left\{-1, -2, -4 \right\}$, Figure \ref{Fig3} represents, as a function of time, the $L^2(\Omega)$-norm of the uncontrolled solution $y^\star(\cdot, t)$, the $L^2(\Omega)$-norm of the linear controlled solution $y_0(\cdot, t)$ (used to initialize the algorithm) and the $L^2(\Omega)$-norm of the controlled solution $y_{k^\star}(\cdot, t)$ obtained by the least-squares algorithm. The linear control $v_0$ associated to $y_0$ and the control $v_{k^\star}$ associated with $y_{k^\star}$ are also depicted.
\begin{itemize}
\item For $c_f=-1$. The uncontrolled solution $y^\star$ oscillates and is bounded. The convergence of the least-squares algorithm is quadratic and is observed after $k^\star = 2$ iterations. The dynamic of the initial state-control pair $(y_0, v_0)$ and the final state-control pair $(y_{k^\star}, v_{k^\star})$ are similar: the non-linearity $f$ therefore has a reduced impact. 
\item For $c_f=-2$. The uncontrolled solution $y^\star$ oscillates more than the previous case and is still bounded. The convergence of the least-squares algorithm is again quadratic and is observed after $k^\star = 3$ iterations. The initial and final dynamics are still similar.
\item For $c_f=-4$. In this case, the uncontrolled solution grows exponentially with respect to the time variable. The algorithm converges again with $k^\star = 5$. The non-linearity $f$ has a strong impact: the controls $v_{0}$ and $v_{k^\star}$ no longer match. In particular, the final control $v_{k^\star}$ acts very strongly at the beginning to balance the exponential growth of the uncontrolled solution.
\end{itemize}
As expected, there is a large gap between the initial control $v_0$ and the final control $v_{k^\star}$ as $|c_f|$ increases.

\begin{table}[!h]
\center
\begin{small}
\begin{tabular}{c c c}
\rotatebox{90}{ \hspace*{2cm} {\large $c_{f} = -1$}} & \includegraphics[scale=0.48]{"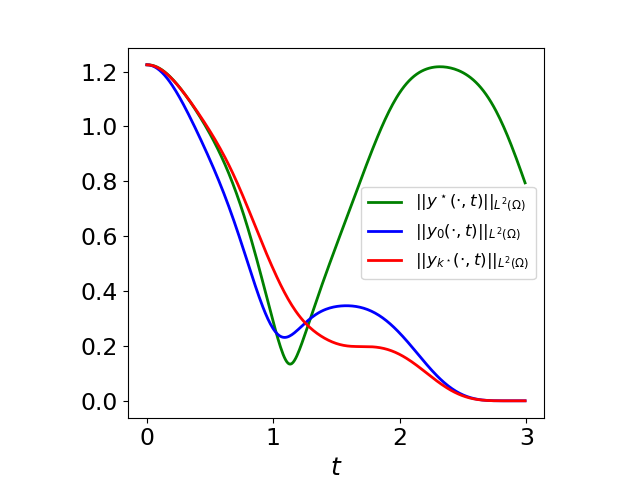"} & \includegraphics[scale=0.48]{"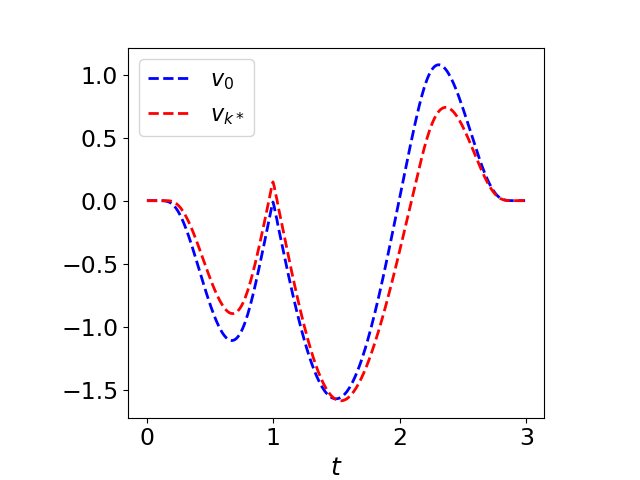"}\\
\rotatebox{90}{ \hspace*{2cm} {\large $c_{f} = -2$}} &\includegraphics[scale=0.48]{"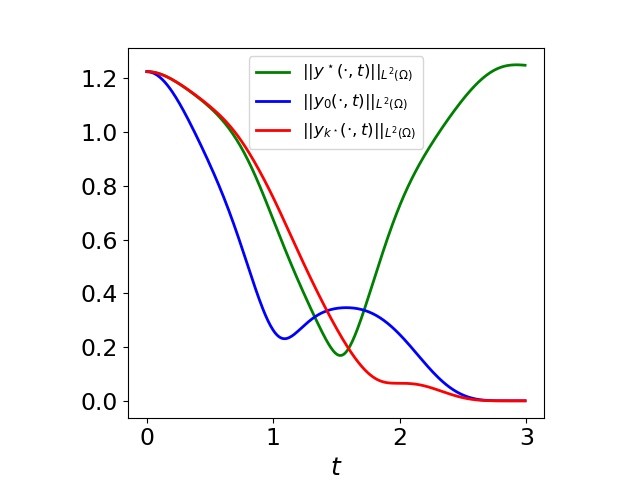"} & \includegraphics[scale=0.48]{"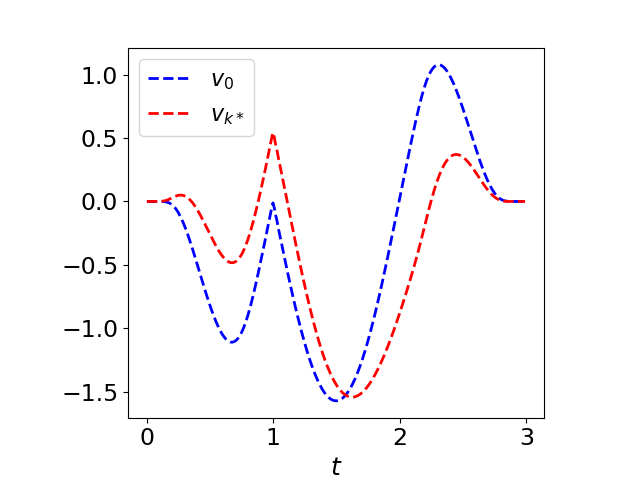"}\\
\rotatebox{90}{ \hspace*{2cm} {\large $c_{f} = -4$}} &\includegraphics[scale=0.48]{"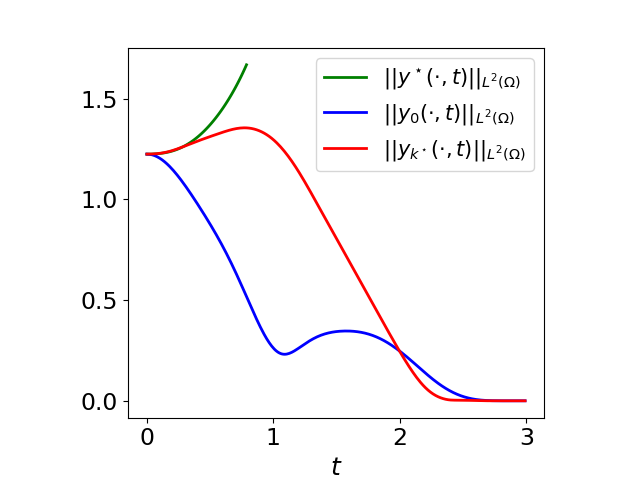"} & \includegraphics[scale=0.48]{"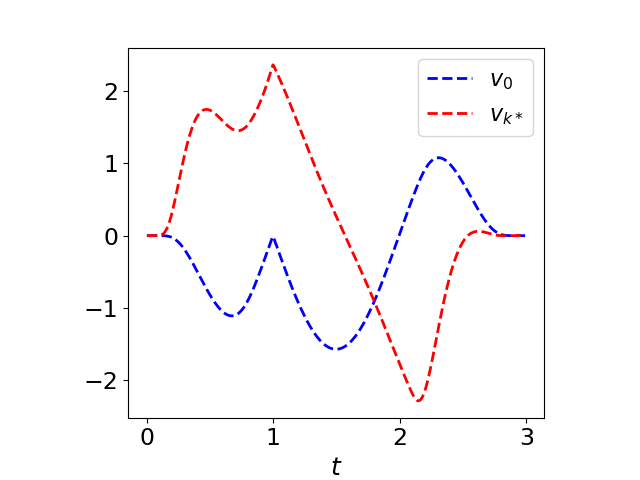"}
\end{tabular}
\end{small}
\captionof{figure}{$c_{u_0} = 1$. \textbf{Right:} Representation of the $\left\|\cdot\right\|_{L^2(\Omega)}$-norm of the uncontrolled solution $y^\star(\cdot, t)$, the linear controlled solution $y_0(\cdot,t)$ used to the initialization and the solution $y_{k^\star}$ obtained by the least-squares algorithm with respect to $t$, for $c_{f}\in \{-1, -2, -4\}$. \textbf{Left:} Representation of the linear control $v_0$ associated with $y_0$ and the control $v_{k^\star}$ associated with $y_{k^\star}$ with respect to $t$, for $c_{f}\in \{-1, -2, -4\}$. }\label{Fig3}
\end{table}

\paragraph{II. Comparison with other algorithms.} 

\subparagraph{II-a. Newton algorithm.} 
When $\lambda_k = 1$ for all $k\in\mathbb{N}$, the least-squares algorithm given by \eqref{LS_algo} coincides with the Newton algorithm $($N$)$ associated with $F: \mathcal{A} \rightarrow L^2(Q_T)$, $F(y,v) := \partial_{tt}y-\partial_{xx}y +f(y)$. In particular, this explains the super-linear convergence property obtained in Theorem \ref{Th_Convergence_results} and numerically illustrated in the first case (see Figure \ref{Fig1}).

We consider $u_0 = 20\left(\cos(\pi x) - 1\right)$ and $u_1 = 0$. For several values of the parameter $c_f$, we compute $(y_{k^\star}, v_{k^\star})_{k\in\mathbb{N}}$ with $\lambda_k = 1$ for all $k\in\mathbb{N}$. Table \ref{ABCD} collects some norms with respect to $k$. With this set of data, we observe that the Newton method converges in fewer iterations than least-squares algorithm (see Table \ref{ABCD} and Table \ref{ABC}).

\begin{table}[!ht]
\center
\begin{small}
\begin{tabular}{|c|c|c|c|c|c|c|c|c|c|}
  \hline
  $c_f$ & $\sqrt{2E(y_{k^\star}, v_{k^\star})}$ & $\|y_{k^\star}\|_{L^2(Q_T)}$ & $\|v_{k^\star}\|_{L^2(0,T)}$ & $\frac{\|y_{k^\star} - y_{k^\star-1}\|_{L^2(Q_T)}}{\|y_{k^\star-1}\|_{L^2(Q_T)}}$ & $\frac{\|v_{k^\star} - v_{k^\star -1}\|_{L^2(0,T)}}{\|v_{k^\star - 1}\|_{L^2(0,T)}}$ & $k^\star$ \\
  \hline
    $10$ & $1.209\times 10^{-9}$ & $4.062\times 10^1$ & $1.883\times 10^3$ & $3.655\times 10^{-7}$ & $3.869\times 10^{-7}$ & $11$ \\
   $5$ & $1.881\times 10^{-11}$ & $3.447\times 10^1$ & $8.185\times 10^2$ & $7.377\times 10^{-8}$ & $4.675\times 10^{-8}$ & $8$ \\
   $2$& $2.549\times 10^{-8}$ & $2.100\times 10^1$ & $9.062\times 10^1$ & $ 8.275\times 10^{-6}$ & $3.373\times 10^{-5}$ & $6$ \\
   $1$ & $6.094\times 10^{-10}$ & $2.087\times 10^1$ & $5.947\times 10^1$ & $ 1.891\times 10^{-6}$ & $1.130\times 10^{-5}$ & $5$ \\
   $-0.5$ & $4.455\times 10^{-6}$ & $4.381\times 10^1$ & $1.418\times 10^2$ & $1.281\times 10^{-4}$ & $2.377\times 10^{-4}$ & $4$\\
  $-1$ & $4.011\times 10^{-6}$ & $1.118\times 10^2$ & $5.879\times 10^2$ & $4.144\times 10^{-5}$ & $1.074\times 10^{-4}$ & $5$\\
   $-1.5$ & $2.321\times 10^{-12}$ & $3.298\times 10^2$ & $2.590\times 10^3$ & $1.967\times 10^{-9}$ & $7.702\times 10^{-9}$ & $6$ \\
   $-2$ & $9.727\times 10^{-12}$ & $9.451\times 10^{2}$ & $1.083\times 10^4$ & $3.816\times 10^{-10}$ & $1.076\times 10^{-9}$ & $7$\\ 
  \hline
\end{tabular}
\end{small}
\caption{$c_{u_0} = 20$ and $\lambda_k = 1$ for all $k\in\mathbb{N}$. Some norms of the solution $y_{k^\star}$ obtained with the least-squares algorithm as well as the associated control $v_{k^\star}$ with respect to the parameter $c_f$. }\label{ABCD}
\end{table}

Now, we fix the parameter $c_f=-2$ and we consider the initial state $(u_0, u_1) =  \left( 50\left(\cos(\pi x) - 1\right), 100 \cdot\mathbf{1}_{(\frac{1}{2}; 1]} \right)$. In particular, we observe that the Newton algorithm (when $\lambda_k = 1$ for all $k\in\mathbb{N}$) diverges: the norms of the solution and control are blowing up as number of iterations are increasing. Remark that the least-squares algorithm converges in $k^\star = 673$ iterations. 

\newpage

\paragraph{II-b. Fixed-point methods.}
 We consider
\begin{equation*}
u_0(x) = 100(x-\frac{1}{2})\cdot\mathbf{1}_{(\frac{1}{2}, 1]}, \hspace{1cm} u_1(x) = 100\cdot\mathbf{1}_{(\frac{1}{2}, 1]}.
\end{equation*}
For several values of $c_f$, we compare the least-squares algorithm with two fixed-point methods. The first one is the fixed-point method associated with the operator $K$ defined by \eqref{Def_fixed_point_operator_K1} in Section \ref{Section_3.2}. This leads to the algorithm \eqref{PF1}. The second one is associated with the operator $\widetilde{K} : L^\infty(Q_T) \rightarrow L^\infty(Q_T)$ defined by $y = \widetilde{K}(\xi)$ where $y$ is a controlled solution of
\begin{equation}\label{tilde_K}
\left\lbrace 
    	\begin{aligned}
    	&\partial_{tt} y - \partial_{xx} y = -f(\xi), &  Q_T, \\ &y(0,\cdot) = 0, \ \ \partial_x y(1,\cdot) = v,   & (0,T) , \\ & \big( y(\cdot,0),\partial_t y(\cdot, 0)\big)  = (u_0, u_1), & \Omega,
    	\end{aligned}
    	\right.
\end{equation} 
given by Theorem \ref{Th.8}. This leads to the following algorithm:
\begin{equation}\tag{$$PF2$$}\label{PF2}
y_0 \in L^\infty(Q_T), \hspace{1cm}  y_{k+1} = \widetilde{K}(y_k), \ k\geq 0.
\end{equation}
Table \ref{Table_1_compare_LS_PF1_PF2}, Table \ref{Table_2_compare_LS_PF1_PF2} and Table \ref{Table_3_compare_LS_PF1_PF2} collect some norms of the sequence $(y_{k},v_{k})_{k\in\mathbb{N}}$ computed by the three algorithms \eqref{LS_algo}, \eqref{PF1} and \eqref{PF2} for respectively $c_f = -0.5$, $c_f = -1$ and $c_f=-2$. Figure \ref{Figure_compare_LS_PF1_PF2} represents the evolution of $\sqrt{2E(y_k, v_k)}$ with respect to the iterations $k$. In particular, $\eqref{PF1}$ and $\eqref{PF2}$ do not usually converge, but if they do, convergence is linear.

\vspace{0.3cm}

\begin{table}[!ht]
\center
\begin{small}
\begin{tabular}{|c|c|c|c|c|c|c|c|c|c|c|}
  \hline
   & $\sqrt{2E(y_{k}, v_{k})}$ & $\|y_{k}\|_{L^2(Q_T)}$ & $\|v_{k}\|_{L^2(0,T)}$ & $\frac{\|y_{k} - y_{k-1}\|_{L^2(Q_T)}}{\|y_{k-1}\|_{L^2(Q_T)}}$ & $\frac{\|v_{k} - v_{k -1}\|_{L^2(0,T)}}{\|v_{k - 1}\|_{L^2(0,T)}}$ & $k^\star$ \\
  \hline
   \eqref{LS_algo} & $2.305\times 10^{-8}$ & $8.394\times 10^1$ & $3.800\times 10^2$ & $2.457\times 10^{-7}$  & $1.446\times 10^{-6}$  & $5$\\ 
   \eqref{PF1} & $7.941\times 10^{-6}$ & $9.614\times 10^1$ & $4.227\times 10^2$ & $2.228\times 10^{-8}$ & $1.878\times 10^{-8}$ & $16$\\ 
   \eqref{PF2} & $9.938\times 10^{-6}$ & $2.338\times 10^2$ & $1.027\times 10^3$ & $2.244\times 10^{-9}$ & $3.254\times 10^{-9}$ & $79$ \\ 
  \hline
\end{tabular}
\end{small}
\caption{$c_f = -0.5$. Some norms obtained by algorithms \eqref{LS_algo}, \eqref{PF1} and \eqref{PF2}.}\label{Table_1_compare_LS_PF1_PF2}
\end{table}
\begin{table}[!ht]
\center
\begin{small}
\begin{tabular}{|c|c|c|c|c|c|c|c|c|c|c|}
  \hline
   & $\sqrt{2E(y_{k}, v_{k})}$ & $\|y_{k}\|_{L^2(Q_T)}$ & $\|v_{k}\|_{L^2(0,T)}$ & $\frac{\|y_{k} - y_{k-1}\|_{L^2(Q_T)}}{\|y_{k-1}\|_{L^2(Q_T)}}$ & $\frac{\|v_{k} - v_{k -1}\|_{L^2(0,T)}}{\|v_{k - 1}\|_{L^2(0,T)}}$ & $k^\star$ \\
  \hline
   \eqref{LS_algo} & $1.518\times 10^{-8}$ & $2.243\times 10^2$ & $1.625\times 10^3$ & $1.344\times 10{-7}$ & $2.183\times 10^{-7}$ & $10$ \\ 
   \eqref{PF1} & $3.917\times 10^{-6}$ & $2.623\times 10^2$ & $1.931\times 10^3$ & $1.478\times 10^{-9}$ & $6.534\times 10^{-10}$ & $28$ \\ 
   \eqref{PF2} & $-$ & $-$ & $-$ & $-$ & $-$ & $+\infty$ \\ 
  \hline
\end{tabular}
\end{small}
\caption{$c_f = -1$. Some norms obtained by algorithms \eqref{LS_algo} and \eqref{PF1}.}\label{Table_2_compare_LS_PF1_PF2}
\end{table}
\begin{table}[!ht]
\center
\begin{small}
\begin{tabular}{|c|c|c|c|c|c|c|c|c|c|c|}
  \hline
   & $\sqrt{2E(y_{k}, v_{k})}$ & $\|y_{k}\|_{L^2(Q_T)}$ & $\|v_{k}\|_{L^2(0,T)}$ & $\frac{\|y_{k} - y_{k-1}\|_{L^2(Q_T)}}{\|y_{k-1}\|_{L^2(Q_T)}}$ & $\frac{\|v_{k} - v_{k -1}\|_{L^2(0,T)}}{\|v_{k - 1}\|_{L^2(0,T)}}$ & $k^\star$ \\
  \hline
   \eqref{LS_algo} & $2.482\times 10^{-8}$ & $1.408\times 10^{3}$ & $1.548\times 10^4$ & $3.769\times 10^{-8}$ & $5.090\times 10^{-8}$ & $76$\\ 
   \eqref{PF1} & $-$ & $-$ & $-$ & $-$ & $-$ & $+\infty$ \\ 
   \eqref{PF2} & $-$ & $-$ & $-$ & $-$ & $-$ & $+\infty$ \\ 
  \hline
\end{tabular}
\end{small}
\caption{$c_f = -2$. Some norms obtained by algorithms \eqref{LS_algo}.}\label{Table_3_compare_LS_PF1_PF2}
\end{table}

\begin{figure}[h!]
\centering
\includegraphics[scale=0.48]{"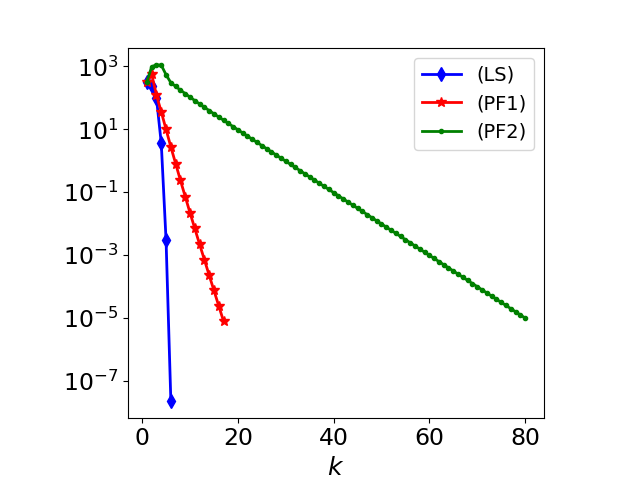"}
\caption{$c_f = -0.5$. Evolution of $\sqrt{2E(y_k, v_k)}$ for \eqref{LS_algo}, \eqref{PF1} and \eqref{PF2} with respect to the iterations $k$.} 
\label{Figure_compare_LS_PF1_PF2}
\end{figure}

\section{Conclusion and perspectives}\label{Section5}

Following \cite{zuazua1993exact}, we give a generalized observability inequality with a constant expressed as an exponential of the potential. Assuming an optimal growth assumption at infinity on the non-linearity of the type $s\ln^2s$, this leads to the existence of a control $v\in L^2(0,T)$ steering the semi-linear system $\eqref{system_semilinear}$ from an initial state $(u_0, u_1)\in H^1_{(0)}(\Omega)\times L^2(\Omega)$ to the target $(z_0, z_1)\in H^1_{(0)}(\Omega)\times L^2(\Omega)$ within time $T$. Under an additional regularity assumption, we adapt the least-squares approach introduced in \cite{munch2022constructive} to boundary case leading to a convergent algorithm. In particular, the convergence is super-linear after a finite number of iterations. 

Numerical experiments are in agreement with the theoretical results. More precisely, the experiments confirm  the change in convergence speed for the least-squares algorithm and suggest that the fixed-point operator $K$ (defined by \eqref{Def_fixed_point_operator_K1}) is not contracting in general.

\vspace{0.3cm}

We conclude with some comments:

\begin{enumerate}

\item \textbf{Construction of a contracting operator.} Numerical simulations show that the fixed-point algorithm \eqref{PF2} associated with $\widetilde{K}$ diverges in general. By introducing a Carleman parameter $s>0$ large enough and weight functions $\rho$, $\rho_1$, we can expect a contraction property for the operator $\mathcal{K} : L^\infty(Q_T) \rightarrow L^\infty(Q_T)$ where $ y:= \mathcal{K} (\xi)$ is the optimal controlled solution of \eqref{tilde_K} for the cost \begin{equation*}
\widetilde{\mathcal{J}}(y,v) = \left\|\rho^{-1}(s)y\right\|_{L^2(Q_T)}^2 + s^{-1}\left\|\rho_1^{-1}(s)v\right\|_{L^2(0,T)}^2.
\end{equation*}
We refer to \cite{bhandari2023exact} and \cite{claret2023exact}.
This requires a Carleman inequality with a Neumann-type observation. 

\item \textbf{The multi-dimensional case.} The generalization of these results in the multi-dimensional case is open. First, in dimension $d>1$, the regularity of the solutions of \eqref{system_linear} depends on the domain $\Omega$ (we refer to \cite[Chapitre III, Section $2$ p.$179$-$180$]{Lions1988ControlabiliteEP} or also \cite[Theorem $1.1$ p.$52$]{LASIECKA198949}). Secondly, to expect a generalization of all these results, we need an observability inequality which holds in any dimensions. Our estimate \eqref{In.obs.g} is based on specific argument related to the one dimension and to our knowledge, there is no estimate like \eqref{In.obs.g} valid in any dimension.


\item \textbf{Inverse problems.} On account of the duality between controllability and observability, it would be interesting to analyze the potential of the least-squares approach for solving inverse problems. 

\end{enumerate}

\bibliographystyle{alpha} 
\bibliography{biblio1}

\end{document}